\documentclass[11pt]{article}
\usepackage{amsfonts, amsmath, amssymb, amscd, amsthm, color, graphicx, mathrsfs, mathabx, wasysym, setspace, mdwlist, calc,float}
\usepackage{setspace}
\usepackage{hyperref}
\usepackage{tikz-cd} 
\usepackage{hyperref}
\usepackage{tocloft}
\usepackage{xcolor}

\usepackage{hyperref}
\hypersetup{linktocpage}

\hypersetup{colorlinks,
    linkcolor={red!50!black},
    citecolor={blue!80!black},
    urlcolor={blue!80!black}}
\usepackage{float}

\usepackage{comment}

\newcommand{\NN}{\mathbb{N}}
\newcommand{\ZZ}{\mathbb{Z}}
\newcommand{\QQ}{\mathbb{Q}}
\newcommand{\RR}{\mathbb{R}}
\newcommand{\CC}{\mathbb{C}}

\newcommand{\inj}{\hookrightarrow}

\newcommand{\tensor}{\overline{\otimes}}

\newtheorem{thm}{Theorem}[section]
\newtheorem{cor}[thm]{Corollary}
\newtheorem{lem}[thm]{Lemma}
\newtheorem{prop}[thm]{Proposition}

\newtheorem{conj}[thm]{Conjecture}
\theoremstyle{definition}
\newtheorem{defn}[thm]{Definition}

\theoremstyle{remark}
\newtheorem{rem}[thm]{Remark}

\newcommand\blfootnote[1]{%
  \begingroup
  \renewcommand\thefootnote{}\footnote{#1}%
  \addtocounter{footnote}{-1}%
  \endgroup
}

\begin{document}

\title{Bi-exactness of relatively hyperbolic groups}
\date{}
\author{Koichi Oyakawa}

\maketitle

\vspace{-10mm}

\begin{abstract}
    We prove that finitely generated relatively hyperbolic groups are bi-exact if and only if all peripheral subgroups are bi-exact. This is a generalization of Ozawa's result which claims that finitely generated relatively hyperbolic groups are bi-exact if all peripheral subgroups are amenable.
\end{abstract}

\section{Introduction}
\blfootnote{\textbf{MSC} Primary: 20F65. Secondary: 20F67, 20F99, 46L10.}
\blfootnote{\textbf{Key words and phrases}: bi-exact groups, relatively hyperbolic groups, proper arrays.}
Bi-exactness is an analytic property of groups defined by Ozawa in \cite{Oz2} (by the name of class $\mathcal{S}$). Recall that a group $G$ is \emph{exact} if there exists a compact Hausdorff space on which $G$ acts topologically amenably and \emph{bi-exact} if it is exact and there exists a map $\mu\colon G\to{\rm Prob}(G)$ such that for every $s,t\in G$, we have
\[\lim_{x\to \infty}\|\mu(sxt)-s.\mu(x)\|_1=0.\]
It is known that exactness is equivalent to Yu's property A (cf. \cite{Oz00}). \par

The notion of bi-exactness is of fundamental importance to the study of operator algebras. Notably, Ozawa proved in \cite{Oz3} that the group von Neumann algebra $L(G)$ of any non-amenable bi-exact icc group $G$ is prime (i.e. it cannot be decomposed as a tensor product of two $\rm{II}_1$ factors) by showing that $L(G)$ is solid. In \cite{OP}, Ozawa and Popa proved a unique prime factorization theorem which states that if $G_1,\cdots,G_n$ are non-amenable bi-exact icc groups and $L(G_1)\tensor\cdots\tensor L(G_n)$ is decomposed as a tensor product of $m$ $\rm{II}_1$ factors with $m\ge n$ i.e. $L(G_1)\tensor\cdots\tensor L(G_n)=\mathcal{N}_1\tensor\cdots\tensor\mathcal{N}_m$, then we have $m=n$ and after permutation of indices and unitary conjugation, each $\mathcal{N}_i$ is isomorphic to an amplification of $L(G_i)$. Subsequently, various rigidity results were proved for bi-exact groups. In particular, Sako showed measure equivalence rigidity for direct product (cf. \cite{Sako}) and Chifan and Ioana showed rigidity result in von Neumann algebra sense for amalgamated free product (cf. \cite{CI}).\par

It is known that the class of bi-exact groups contains amenable groups, hyperbolic groups, discrete subgroups of connected simple Lie groups of rank one, and $\ZZ^2 \rtimes SL_2(\ZZ)$ (cf. \cite{BO}\cite{Oz2}\cite{Oz09}). On the other hand, unlike exactness, bi-exactness is not preserved under some basic group theoretic constructions. For example, direct products and increasing unions of bi-exact groups are not necessarily bi-exact (e.g. $F_2\times F_2$ and $G\times F_2$, where $F_2$ is a free group of rank 2 and $G$ is an infinite countable locally finite group, are exact but not bi-exact). \par
The standard way of proving bi-exactness is using topologically amenable action on a special compact space called `boundary small at infinity'. Using this method, Ozawa proved in \cite{Oz1} that any finitely generated group hyperbolic relative to amenable peripheral subgroups is bi-exact. In this paper, we use a different approach and generalize Ozawa's result by proving the following.

\begin{thm}\label{Thm:main}
Suppose that $G$ is a finitely generated group hyperbolic relative to a collection of subgroups $\{H_\mu\}_{\mu \in \Lambda}$ of $G$. Then, $G$ is bi-exact if and only if all subgroups $H_\mu$ are bi-exact.
\end{thm}

The `Only if' direction is obvious and the whole paper is devoted to proving the `if' direction. In relation to this result, it is worth mentioning the following folklore conjecture.

\begin{conj}\label{conj:folklore}
Suppose that a finitely generated group $G$ is hyperbolic relative to a collection of subgroups $\{H_\mu\}_{\mu \in \Lambda}$ of $G$. If all subgroups $H_\mu$ are exact, then $G$ is bi-exact relative to $\{H_\mu\}_{\mu \in \Lambda}$.
\end{conj}

For the definition of relative bi-exactness, the reader is referred to Definition 15.1.2 of \cite{BO}. Note that the `if' direction of Theorem \ref{Thm:main} is not a weak version of Conjecture \ref{conj:folklore}. Indeed, while the assumption of Theorem \ref{Thm:main} is stronger, the conclusion is also stronger. \par

Our proof of Theorem \ref{Thm:main} is based on a characterization of bi-exactness using exactness and the existence of a proper array (cf. Proposition \ref{Prop:equiv cond of bi-exact}) and uses two technologies related to relatively hyperbolic groups. The first one is a bicombing of fine hyperbolic graphs, which was constructed by Mineyev and Yaman in \cite{MY} using the same idea as in \cite{Mi}. The second one is based on the notion of separating cosets of hyperbolically embedded subgroups, which was introduced by Hull and Osin in \cite{HO} and developed further by Osin in \cite{Os2}. 
We will construct two arrays using each of these techniques and combine them to make a proper array. It is worth noting that the Mineyev-Yaman's bicombing of relatively hyperbolic groups alone is not sufficient to derive Conjecture \ref{conj:folklore} for the reason explained in Remark \ref{Rem:not proper}. Therefore, Conjecture \ref{conj:folklore} is still considered open.
\par
The paper is organized as follows. In Section \ref{Section: preliminary}, we discuss the necessary definitions and known results about bi-exact groups and relatively hyperbolic groups. In Section \ref{Subsection:overview}, we give an outline of our proof. Section \ref{Subsection:first array} and \ref{Subsection:second array} discuss the constructions of arrays based on ideas of \cite{MY} and \cite{HO}, respectively, and Section \ref{Subsection:proof of main thm} provides the proof of Theorem~\ref{Thm:main}.
\par

\vspace{2mm}

\noindent\textbf{Acknowledgment.}
I thank Denis Osin for introducing this topic to me, for sharing his ideas, and for many helpful discussions.

\section{Preliminary}\label{Section: preliminary}
\subsection{Bi-exact groups}
In this section, we introduce some equivalent conditions of bi-exact groups. The definition of an array given below was suggested in \cite{CSU}.

\begin{defn}\label{Def:array}
Suppose that $G$ is a group, $\mathcal{K}$ a Hilbert space, and $\pi\colon G\to \mathcal{U}(\mathcal{K})$ a unitary representation. A map $r\colon G\to \mathcal{K}$ is called an \emph{array} on $G$ into $(\mathcal{K},\pi)$, if $r$ satisfies (1) and (2) below. When there exists such $r$, we say that $G$ admits an array into $(\mathcal{K},\pi)$.
\begin{itemize}
    \item [(1)]
    $\pi_g(r(g^{-1}))=-r(g)$ for all $g\in G$.
    \item [(2)]
    For every $g\in G$, we have $\sup_{h\in G}\|r(gh)-\pi_g(r(h))\|<\infty$.
\end{itemize}
If, in addition, $r$ satisfies (3) below, $r$ is called \emph{proper}.
\begin{itemize}
    \item [(3)]
    For any $N\in \NN$, $\{g\in G\mid \|r(g)\|\le N\}$ is finite.
\end{itemize}
\end{defn}

Conditions (2) and (3) in Proposition \ref{Prop:equiv cond of bi-exact} are simplified versions of Proposition 2.3 of \cite{CSU} and Proposition 2.7 of \cite{PV} respectively.

\begin{prop}\label{Prop:equiv cond of bi-exact}
For any countable group $G$, the following three conditions are equivalent.
\begin{itemize}
    \item[(1)]
    $G$ is bi-exact.
    \item[(2)]
    $G$ is exact and admits a proper array into the left regular representation $(\ell^2(G),\lambda_G)$.
    \item[(3)]
    $G$ is exact, and there exist an orthogonal representation $\eta\colon G \to \mathcal{O}(K_\RR)$ to a real Hilbert space $K_\RR$ that is weakly contained in the regular representation of $G$ and a map $c\colon G\to K_\RR$ that is proper and satisfies
\[\sup_{k\in G}\|c(gkh)-\eta_g c(k)\|<\infty\] 
    for all $g,h\in G$.
\end{itemize}
\end{prop}

\subsection{Relatively hyperbolic groups}
There are several equivalent definitions of relatively hyperbolic groups due to Gromov, Bowditch, Farb, and Osin. We adopt Farb's definition with the Bounded Coset Penetration property replaced by fineness of coned-off Cayley graphs. \cite{Hr}, \cite[Appendix]{Da}, and \cite[Appendix]{Os1} contain proofs of the equivalence of these definitions. To state our definition, we first prepare some terminologies. When we consider a graph $Y$ without loops or multiple edges, the edge set $E(Y)$ is defined as a irreflexive symmetric subset of $V(Y)\times V(Y)$, where $V(Y)$ is the vertex set (see Section \ref{Section:MY} for details).

\begin{defn}[Coned-off Cayley graph]
\label{coned off cayley graph}
Suppose that $G$ is a finitely generated group with a finite generating set $X$ and $\{H_\lambda\}_{\lambda\in\Lambda}$ is a collection of subgroups of $G$. Let $\Gamma(G,X)$ be the Cayley graph of $G$ with respect to $X$, in which we don't allow loops nor multiple edges. More precisely, its vertex set and edge set are defined by
\begin{equation}\label{Eq: Cayley graph}
    V(\Gamma(G,X))=G,\;\;\;
E(\Gamma(G,X))=\{(g,gs),(gs,g) \mid g\in G, s\in X\setminus\{1\}\}.
\end{equation}
We define a new graph $\widehat{\Gamma}$ whose vertex set $V(\widehat{\Gamma})$ and edge set $E(\widehat{\Gamma})$ are given by
\[V(\widehat{\Gamma})=G\sqcup \bigsqcup_{\lambda\in\Lambda}G/H_\lambda,\] \[E(\widehat{\Gamma})=E(\Gamma(G,X))\sqcup \bigsqcup_{\lambda\in\Lambda}\{(g,xH_\lambda),(xH_\lambda,g)\mid xH_\lambda\in G/H_\lambda\ , g\in xH_\lambda\}.\]
The graph $\widehat{\Gamma}$ is called \emph{coned-off Cayley graph} with respect to $X$.
\end{defn}

We set the length of any edge of $\widehat{\Gamma}$ to be 1. $G$ acts on $\widehat{\Gamma}$ naturally by graph automorphism. For a vertex $v\in V(\widehat{\Gamma})$, we denote its stabilizer by $G_v=\{g\in G\mid gv=v\}$.

\begin{defn}[Hyperbolic space]
A geodesic metric space $(X,d)$ is called \emph{hyperbolic} if there exists $\delta\in[0,\infty)$ such that for any $x,y,z\in X$, any geodesic path $[x,y],[x,z],[z,y]$, and any point $a\in[x,y]$, there exists $b\in[x,z]\cup[z,y]$ satisfying $d(a,b)\le\delta$.
\end{defn}

\begin{defn}[Circuit]
A closed path $p$ in a graph is called \emph{circuit}, if $p$ doesn't have self-intersection except its initial and terminal vertices.
\end{defn}

\begin{defn}[Fine graph]
A graph $Y$ is called \emph{fine}, if for any $n\in\NN$ and any edge $e$ of $Y$, the number of circuits of length $n$ containing $e$ is finite.
\end{defn}

The following is the definition of relatively hyperbolic groups on which we work. The reader is referred to \cite[Appendix]{Da} for the equivalence of Definition \ref{Def:rel hyp gp} and the other definitions.

\begin{defn}\label{Def:rel hyp gp}
A finitely generated group $G$ is called \emph{hyperbolic relative to} a collection of subgroups $\{H_\lambda\}_{\lambda\in\Lambda}$, if $\Lambda$ is finite and for some (equivalently, any) finite generating set $X$ of $G$, the coned-off Cayley graph $\widehat{\Gamma}$ is fine and hyperbolic. A member of the collection of subgroups $\{H_\lambda\}_{\lambda\in\Lambda}$ is called a \emph{peripheral subgroup}.
\end{defn}

Some definitions (e.g. Osin's definition in \cite{Os1}) can be stated without requiring $\Lambda$ to be finite. In Osin's definition, finiteness of $\Lambda$ follows if $G$ is finitely generated.

\subsection{Mineyev-Yaman's bicombing}\label{Section:MY}
In \cite{MY}, Mineyev and Yaman constructed a bicombing of a fine hyperbolic graph using the same idea as in \cite{Mi}. We first discuss 1-chains on a graph and a unitary representation of a group acting on a graph in general and then introduce Mineyev-Yaman's bicombing. \par
In this section, we consider graphs without loops nor multiple edges, following \cite{MY} and \cite{Bow}. More precisely, for a graph $Y$ with a vertex set $V(Y)$, its edge set $E(Y)$ is a subset of $V(Y)\times V(Y)\setminus\{(v,v)\mid v\in V(Y)\}$ such that $\overline{E(Y)}=E(Y)$, where we define $\overline{(u,v)}=(v,u)$ for any $(u,v)\in V(Y)^2$. We call a subset $E^+(V)$ of $E(Y)$ a set of {\it positive edges} if it satisfies $E(V)=E^+(V)\sqcup \overline{E^+(V)}$. Choosing a set of positive edges $E^+(V)$ means choosing directions of edges. Note that when we consider a group action on a graph, we allow inversion of edges. We begin with auxiliary notations.\par

For a graph $Y$, we denote by $C_1(Y)$ the set of 1-chains on $Y$ over $\CC$. More precisely, $C_1(Y)$ is defined as follows. Consider a direct product $\CC^{E(Y)}=\Big\{\sum_{e\in E(Y)}c_ee \;\Big|\; c_e\in\CC\Big\}$ of vector spaces over $\CC$. Note that this summation notation is just a formal sum. We define a quotient space $\ell^0(Y)$ by
\[\ell^0(Y)=\CC^{E(Y)} \;\Big/\; \Big\{\sum_{e\in E(Y)}c_ee\in \CC^{E(Y)} \; \Big| \; c_e=c_{\bar{e}} \; \forall e\in E(Y)\Big\}.\]
Here, we denote by $[e]=[u,v]$ the element in $\ell^0(Y)$ corresponding to $e=(u,v)\in \CC^{E(Y)}$ where $e=(u,v)\in E(Y)$. Note that for any $e=(u,v)\in E(Y)$, we have 
\[[\bar{e}]=[v,u]=-[u,v]=-[e].\]\par

Fix a set of positive edges $E^+(Y)$. The map
\[\CC^{E^+(Y)}\ni \sum_{e\in E^+(Y)}c_{e}e\mapsto
\sum_{e\in E^+(Y)}c_{e}[e]\in \ell^0(Y)\]
is an isomorphism of vector spaces. The space of 1-chains is defined by
\[C_1(Y)=\Big\{\sum_{e\in E^+(Y)}c_{e}[e]\in \ell^0(Y)\;\Big|\;
\exists F\subset E^+(Y),\; |F|<\infty \;\wedge\; c_e=0 \;\forall e\notin F\Big\}.\]
Note that $C_1(Y)$ is independent of the choice of $E^+(Y)$. For $a,b\in V(Y)$ and a path $q=(v_0,v_1,\cdots,v_n)$ from $a=v_0$ to $b=v_n$, we define a 1-chain 
\begin{equation}\label{Eq:path}
q=[v_0,v_1]+\cdots+[v_{n-1},v_n]\in C_1(Y).
\end{equation}

Here, we use the same notation $q$ to denote a path and the corresponding 1-chain by abuse of notation.

For each $p\in[1,\infty)$ (in fact, we use only cases $p=1, 2$), define a subspace $\ell^p(Y)$ of $\ell^0(Y)$ by
\[\ell^p(Y)=\Big\{\sum_{e\in E^+(Y)}c_{e}[e]\in \ell^0(Y)\;\Big|\; \sum_{e\in E^+(Y)}|c_{e}|^p<\infty\Big\}
\]
and for each $\xi=\sum_{e\in E^+(Y)}c_{e}[e]\in\ell^p(Y)$ define a norm \[\|\xi\|_p=\left(\sum_{e\in E^+(Y)}|c_{e}|^p\right)^{1/p}.\]  $(\ell^p(Y),\|\cdot\|_p)$ becomes a Banach space isomorphic to $\ell^p(E^+(Y))$. In particular, $(\ell^2(Y),\|\cdot\|_2)$ is a Hilbert space. Note that $(\ell^p(Y),\|\cdot\|_p)$ doesn't depend on the choice of $E^+(Y)$. We also have $C_1(Y)\subset\ell^p(Y)$ for any $p\in[1,\infty)$. \par

\begin{defn}\label{Def:unitary rep}
Suppose that a group $G$ acts on $Y$ by graph automorphisms, then $G$ acts on $(\ell^p(Y),\|\cdot\|_p)$ isometrically by
\[G\times \ell^p(Y)\ni \Big(g, \sum_{(u,v)\in E^+(Y)}c_{(u,v)}[u,v]\Big)\mapsto
\sum_{(u,v)\in E^+(Y)}c_{(u,v)}[gu,gv]\in \ell^p(Y).
\]
In particular, this action on $(\ell^2(Y),\|\cdot\|_2)$ becomes a unitary representation of $G$ and we denote this unitary representation by $(\ell^2(Y), \pi)$.
\end{defn}

There exists a unique well-defined linear map 
\[
\partial\colon C_1(Y)\to C_0(Y)
=\Big\{ \sum_{v\in F}c_vv \mid F\subset V(Y),\; |F|<\infty,\; c_v \in \CC \Big\}
\]
such that $\partial[u,v]=v-u$ for any $(u,v)\in E(Y)$.

\begin{defn}\label{Def:bicombing}
Suppose that $Y$ is a graph. A map $q\colon V(Y)^2\to C_1(Y)$ is called a \emph{homological bicombing}, if for any $(a,b)\in V(Y)^2$, we have $\partial q[a,b]=b-a$. In particular, if all coefficients of $q[a,b]$ are in $\QQ$ for any $a,b\in V(Y)$, $q$ is called a \emph{$\QQ$-bicombing}.
\end{defn}

\begin{defn}[2-vertex-connectivity]
A graph $Y$ is called \emph{2-vertex-connected}, if $Y$ is connected and for any $v\in V(Y)$, $Y\setminus\{v\}$ is connected, where $Y\setminus\{v\}$ is an induced subgraph of $Y$ whose vertex set is $V(Y)\setminus\{v\}$.
\end{defn}

Theorem \ref{Thm:bicombing} is a simplified version of Theorem 47 and Proposition 46 (3) in \cite{MY}.

\begin{thm}\label{Thm:bicombing}
Suppose that $Y$ is a 2-vertex-connected fine hyperbolic graph and $G$ is a group acting on $Y$. If the number of $G$-orbits of $E(Y)$ is finite and the edge stabilizer $G_e=G_u\cap G_v$ is finite for any $e=(u,v)\in E(Y)$, then there exists a $\QQ$-bicombing $q$ of $Y$ satisfying the following conditions.
\begin{itemize}
    \item[(1)]
    $q$ is $G$-equivariant, i.e. $q[ga,gb]=g\cdot q[a,b]$ for any $a,b\in V(Y)$ and any $g\in G$.
    \item[(2)]
    $q$ is anti-symmetric, i.e. $q[b,a]=-q[a,b]$ for any $a,b\in V(Y)$.
    \item[(3)]
    There exists a constant $T\ge0$ such that for any $a,b,c\in V(Y)$,
    \[\|q[a,b]+q[b,c]+q[c,a]\|_1\le T.\]
    \item[(4)]
    There exist constants $M'\ge0$ and $N'\ge0$ such that for any $a,b\in V(Y)$,
    \[\|q[a,b]\|_1\le M'd_Y(a,b)+N',\]
    where $d_Y$ is the graph metric of $Y$.
    \end{itemize}
\end{thm}

\begin{rem}\label{Rem:convex}
By examining the explicit construction of $q$ in Section 6 of \cite{MY}, we can see that for any $a,b\in V(Y)$, $q[a,b]$ is a convex combination of paths from $a$ to $b$ (see (\ref{Eq:path})), i.e. there exist paths $p_1,\cdots,p_n$ from $a$ to $b$ and $\alpha_1,\cdots,\alpha_n\in\QQ_{\ge0}$ with $\sum_{j=1}^n \alpha_j=1$ such that
\[q[a,b]=\sum_{j=1}^n \alpha_jp_j.\]
\end{rem}

\begin{rem}
Theorem \ref{Thm:bicombing} was actually proved for what Mineyev and Yaman called a `hyperbolic tuple' which was defined in their paper as a quadruple consisting of a graph, a group acting on it, a subset of the vertex set, and a set of vertex stabilizers (cf. \cite[Definitions 20, 27, 29, 38]{MY}). However, the statement and the proof of Theorem 47 and Proposition 46 (3) in \cite{MY} does not require the whole quadruple, and we simplify their exposition.
\end{rem}

\begin{rem}
Theorem \ref{Thm:bicombing} (4) is necessary only for Remark \ref{Rem:not proper}.
\end{rem}

\subsection{Separating cosets of hyperbolically embedded subgroups}
\label{Subsection:HO}
In this section, we define Hull-Osin's separating cosets and explain their properties required for our discussion in Section \ref{Subsection:second array}. The notion of separating cosets of hyperbolically embedded subgroups was first introduced by Hull and Osin in \cite{HO} and further developed by Osin in \cite{Os2}. There is one subtle difference in the definition of separationg cosets in \cite{HO} and in \cite{Os2}, which is explained in Remark \ref{difference}, though other terminologies and related propositions are mostly the same between them. With regards to this difference, we follow definitions and notations of \cite{Os2} in our discussion. In this section, suppose that $G$ is a group, $X$ is a subset of $G$, and $\{H_\lambda\}_{\lambda\in\Lambda}$ is a collection of subgroups of $G$ such that $X\cup (\bigcup_{\lambda\in\Lambda}H_\lambda)$ generates $G$. Note that $\Lambda$ and $X$ are possibly infinite. We begin with defining some auxiliary concepts. \par

Let $\mathcal{H}= \bigsqcup_{\lambda \in \Lambda}(H_\lambda \setminus \{1\})$. Note that this union is disjoint as sets of labels, not as subsets of $G$. Let $\Gamma(G, X\cup\mathcal{H})$ be the Cayley graph of $G$ with respect to $X\sqcup\mathcal{H}$, which allows loops and multiple edges, that is, its vertex set is $G$ and its positive edge set is $G\times(X\sqcup\mathcal{H})$. We call $\Gamma(G, X\cup\mathcal{H})$ \emph{the relative Cayley graph}. \par 
For each $\lambda\in\Lambda$, we consider the Cayley graph $\Gamma(H_\lambda,H_\lambda\setminus\{1\})$, which is a subgraph of $\Gamma(G, X\cup\mathcal{H})$, and define a metric $\widehat{d_\lambda}\colon H_\lambda \times H_\lambda \to [0,\infty]$ as follows. We say that a path $p$ in $\Gamma(G, X\cup\mathcal{H})$ is $\lambda$-\emph{admissible}, if $p$ doesn't contain any edge of $\Gamma(H_\lambda,H_\lambda\setminus\{1\})$. Note that $p$ can contain an edge whose label is an element of $H_\lambda$ (e.g. the case when the initial vertex of the edge is not in $H_\lambda$) and also $p$ can pass vertices of $\Gamma(H_\lambda,H_\lambda\setminus\{1\})$. For $f,g\in H_\lambda$, we define $\widehat{d_\lambda}$ to be the minimum of lengths of all $\lambda$-admissible paths from $f$ to $g$. If there is no $\lambda$-admissible path from $f$ to $g$, then we define $\widehat{d_\lambda}(f,g)=\infty$. For convenience, we extend $\widehat{d_\lambda}$ to $\widehat{d_\lambda}\colon  G\times G \to [0,\infty]$ by defining $\widehat{d_\lambda}(f,g)=\widehat{d_\lambda}(1,f^{-1}g)$ if $f^{-1}g\in H_\lambda$ and $\widehat{d_\lambda}(f,g)=\infty$ otherwise. \par

\begin{defn}\label{Def:hyp emb}
Suppose that $G$ is a group and $X$ is a subset of $G$. The collection of subgroups $\{H_\lambda\}_{\lambda\in\Lambda}$ of $G$ is said to be \emph{hyperbolically embedded in} $(G,X)$, if it satisfies the two conditions below.
\begin{itemize}
    \item [(1)] $X\cup (\bigcup_{\lambda\in\Lambda}H_\lambda)$ generates $G$ and $\Gamma(G, X\cup\mathcal{H})$ is hyperbolic.
    \item [(2)] For any $\lambda\in\Lambda$, $(H_\lambda,\widehat{d_\lambda})$ is locally finite, i.e. for any $n\in\NN$, 
    $\{g\in H_\lambda \mid \widehat{d_\lambda}(1,g)\le n\}$ is finite.
\end{itemize}
\end{defn}

The following is a simplified version of Proposition 4.28 of \cite{DGO}. 
\begin{prop}\label{Prop:hyp emb}
Suppose that $G$ is a finitely generated group hyperbolic relative to a collection of subgroups $\{H_\mu\}_{\mu \in \Lambda}$ of $G$. Then,
for any finite generating set $X$ of $G$, $\{H_\lambda\}_{\lambda\in\Lambda}$ is hyperbolically embedded in $(G,X)$.
\end{prop}

In the remainder of this section, suppose that $\{H_\lambda\}_{\lambda\in\Lambda}$ is hyperbolically embedded in $(G,X)$.

\begin{defn}\cite[Definition 4.1]{Os2}
Suppose that $p$ is a path in the relative Cayley graph $\Gamma(G, X\cup\mathcal{H})$. A subpath $q$ of $p$ is called an \emph{$H_\lambda$-subpath} if the labels of all edges of $q$ are in $H_\lambda$. In the case $p$ is a closed path, $q$ can be a subpath of any cyclic shift of $p$. An $H_\lambda$-subpath $q$ of a path $p$ is called \emph{$H_\lambda$-component} if $q$ is not contained in any longer $H_\lambda$-subpath of $p$. In the case $p$ is a closed path, we require that $q$ is not contained in any longer $H_\lambda$-subpath of any cyclic shift of $p$. Further, by a \emph{component}, we mean an $H_\lambda$-component for some $H_\lambda$. Two $H_\lambda$-components $q_1,q_2$ of a path $p$ is called \emph{connected}, if all vertices of $q_1$ and $q_2$ are in the same $H_\lambda$-coset. An $H_\lambda$-component $q$ of a path $p$ is called \emph{isolated}, if $q$ is not connected to any other $H_\lambda$-component of $p$.
\end{defn}

\begin{rem}
Note that all vertices of an $H_\lambda$-component lie in the same $H_\lambda$-coset.
\end{rem}

The following proposition is important, which is a particular case of Proposition 4.13 of \cite{DGO}.

\begin{prop}\emph{\cite[Lemma 2.4]{HO}}\label{Prop:2.4 of HO}
There exists a constant $C>0$ such that for any geodesic $n$-gon $p$ in $\Gamma(G,X\cup\mathcal{H})$ and any isolated $H_\lambda$-component $a$ of $p$, we have
\[\widehat{d_\lambda}(a_-,a_+)\le nC.\]
\end{prop}

In what follows, we fix any constant $D>0$ with
\begin{equation}\label{Eq:D>3C}
D \ge 3C.
\end{equation}

We can now define separating cosets.

\begin{defn}\cite[Definition 4.3]{Os2}
\label{separating cosets}
A path $p$ in $\Gamma(G,X\cup\mathcal{H})$ is said to \emph{penetrate} a coset $xH_\lambda$ for some $\lambda\in\Lambda$, if $p$ decomposes as $p_1ap_2$, where $p_1,p_2$ are possibly trivial, $a$ is an $H_\lambda$-component, and $a_-\in xH_\lambda$. Note that if $p$ is a geodesic, $p$ penetrates any coset of $H_\lambda$ at most once. In this case, $a$ is called the \emph{component of} $p$ \emph{corresponding to} $xH_\lambda$ and also the vertices $a_-$ and $a_+$ are called the \emph{entrance and exit points} of $p$ and are denoted by $p_{in}(xH_\lambda)$ and $p_{out}(xH_\lambda)$ respectively. If in addition, we have $\widehat{d_\lambda}(a_-,a_+)>D$, then we say that $p$ \emph{essentially penetrates} $xH_\lambda$. For $f,g\in G$ and $\lambda\in\Lambda$, if there exists a geodesic path from $f$ to $g$ in $\Gamma(G,X\cup\mathcal{H})$ which essentially penetrates an $H_\lambda$-coset $xH_\lambda$, then $xH_\lambda$ is called an $(f,g;D)$-\emph{separating coset}. We denote the set of $(f,g;D)$-separating $H_\lambda$-cosets by $S_\lambda(f,g;D)$.
\end{defn}

\begin{rem}\label{difference}
In Definition 3.1 of \cite{HO}, whenever $f,g\in G$ are in the same $H_\lambda$-coset $xH_\lambda$ for some $\lambda\in\Lambda$, $xH_\lambda$ is included in $S_\lambda(f,g;D)$, but in our Definition \ref{separating cosets}, $S_\lambda(f,g;D)$ can be empty even in this case.
\end{rem}

The following lemma is immediate from the above definition.

\begin{lem}\emph{\cite[Lemma 3.2]{HO}}\label{Lem:3.2 of HO}
For any $f,g,h\in G$ and any $\lambda\in\Lambda$, the following holds.
\begin{itemize}
    \item [(a)]
    $S_\lambda(f,g;D)=S_\lambda(g,f;D)$.
    \item[(b)]
    $S_\lambda(hf,hg;D)=\{hxH_\lambda \mid xH_\lambda\in S_\lambda(f,g;D)\}$.
\end{itemize}
\end{lem}

We state some nice properties of separating cosets, all of which were proven in \cite{HO}. For $f,g\in G$, we denote by $\mathcal{G}(f,g)$ the set of all geodesic paths in $\Gamma(G,X\cup\mathcal{H})$ from $f$ to $g$.

\begin{lem}\emph{\cite[Lemma 3.3]{HO}}\label{Lem:3.3 of HO}
For any $\lambda\in\Lambda$, any $f,g\in G$, and any $(f,g;D)$-separating coset $xH_\lambda$, the following holds.
\begin{itemize}
    \item [(a)]
    Every path in $\Gamma(G,X\cup\mathcal{H})$ connecting $f$ to $g$ and composed of at most 2 geodesics penetrates $xH_\lambda$.
    \item [(b)]
    For any $p,q\in \mathcal{G}(f,g)$, we have
    \[\widehat{d_\lambda}(p_{in}(xH_\lambda),q_{in}(xH_\lambda))\le3C \;\;and\;\;
    \widehat{d_\lambda}(p_{out}(xH_\lambda),q_{out}(xH_\lambda))\le3C.\]
\end{itemize}
\end{lem}

Proof of Lemma \ref{Lem:3.3 of HO} is the same as Lemma 3.3 of \cite{HO}, though their statements are slightly different. Actually, in Lemma \ref{Lem:3.3 of HO}, we don't need to assume $f^{-1}g\notin H_\lambda$. This difference comes from the difference of definitions of separating cosets, which was mentioned in Remark \ref{difference}.

\begin{cor}\emph{\cite[Corollary 3.4]{HO}}
\label{Cor:separating cosets are finite}
For any $f,g\in G$ and any $p\in \mathcal{G}(f,g)$, we have $S_\lambda(f,g;D)\subset P_\lambda(p)$, where we define $P_\lambda(p)$ to be the set of all $H_\lambda$-cosets which $p$ penetrates. In particular, we have $|S_\lambda(f,g;D)|\le d_{X\cup\mathcal{H}}(f,g)$, hence $S_\lambda(f,g;D)$ is finite.
\end{cor}

The following lemma makes $S_\lambda(f,g;D)$ into a totally ordered set.

\begin{lem}\emph{\cite[Lemma 3.5]{HO}}
Suppose $f,g\in G$ and that $p\in \mathcal{G}(f,g)$ penetrates an $H_\lambda$-coset $xH_\lambda$ and decomposes as $p=p_1ap_2$, where $p_1,p_2$ are possibly trivial and $a$ is an $H_\lambda$-component corresponding to $xH_\lambda$. Then, we have $d_{X\cup\mathcal{H}}(f,a_-)=d_{X\cup\mathcal{H}}(f,xH_\lambda)$.
\end{lem}

\begin{defn}\cite[Definition 3.6]{HO}
Given any $f,g\in G$, a relation $\preceq$ on the set $S_\lambda(f,g;D)$ is defined as follows.
\[xH_\lambda \preceq yH_\lambda \iff d_{X\cup\mathcal{H}}(f,xH_\lambda)\le d_{X\cup\mathcal{H}}(f,yH_\lambda).\]
\end{defn}

Next, we define a set of pairs of entrance and exit points of a separating coset. That is, for $f,g\in G$, $\lambda\in\Lambda$, and $xH_\lambda\in S_\lambda(f,g;D)$, we define
\[E(f,g;xH_\lambda,D)=\{(p_{in}(xH_\lambda),p_{out}(xH_\lambda))\mid p\in \mathcal{G}(f,g)\}.\]
Note that because $xH_\lambda$ is a $(f,g;D)$-separating coset, any geodesic from $f$ to $g$ penetrates $xH_\lambda$ by Lemma \ref{Lem:3.3 of HO} (a).

\begin{lem}\emph{\cite[Lemma 3.8]{HO}}
\label{Lem:3.8 of HO}
For any $f,g,h\in G$, $\lambda\in\Lambda$, and $xH_\lambda\in S_\lambda(f,g;D)$, the following holds.
\begin{itemize}
    \item [(a)]
    $E(g,f;xH_\lambda,D)=\{(v,u)\mid (u,v)\in E(f,g;xH_\lambda,D)\}$.
    \item [(b)]
    $E(hf,hg;hxH_\lambda,D)=\{(hu,hv)\mid (u,v)\in E(f,g;xH_\lambda,D)\}$.
    \item [(c)]
    $|E(f,g;xH_\lambda,D)|<\infty$.
\end{itemize}
\end{lem}

\begin{lem}
For any $f,g,h\in G$ and any $xH_\lambda\in S_\lambda(f,g;D)$, $xH_\lambda$ is either penetrated by all geodesics from $f$ to $h$ or penetrated by all geodesics from $h$ to $g$.
\end{lem}

\begin{proof}
Suppose there exists a geodesic $q\in\mathcal{G}(f,h)$ which doesn't penetrate $xH_\lambda$. For any geodesic $r\in\mathcal{G}(h,g)$, we can apply Lemma \ref{Lem:3.3 of HO} (a) to a path $qr$ and conclude $r$ penetrates $xH_\lambda$.
\end{proof}

The following lemma is crucial in constructing an array that satisfies the bounded area condition.

\begin{lem}\emph{\cite[Lemma 3.9]{HO}}\label{Lem:3.9 of HO}
For any $f,g,h\in G$ and any $\lambda\in\Lambda$, $S_\lambda(f,g;D)$ can be decomposed as $S_\lambda(f,g;D)=S'\sqcup S''\sqcup F$ where
\begin{itemize}
    \item [(a)]
    $S'\subset S_\lambda(f,h;D)\setminus S_\lambda(h,g;D)$ and we have $E(f,g;xH_\lambda,D)=E(f,h;xH_\lambda,D)$ for any $xH_\lambda\in S'$,
    \item [(b)]
    $S''\subset S_\lambda(h,g;D)\setminus S_\lambda(f,h;D)$ and we have $E(f,g;xH_\lambda,D)=E(h,g;xH_\lambda,D)$ for any $xH_\lambda\in S''$,
    \item [(c)]
    $|F|\le2$.
\end{itemize}
\end{lem}

\section{Main theorem}\label{Section:main section}

\subsection{Overview of proof}\label{Subsection:overview}
In this section, we explain the idea of the construction of a proper array in the proof of Proposition \ref{Prop:main result} (see (\ref{Eq:main rep}) (\ref{Eq:main array})) by considering two particular cases. \par

First, suppose that $G$ is a hyperbolic group i.e., the collection of peripheral subgroups is empty. We take a symmetric finite generating set $X_0$ of $G$ with $1\in X_0$ and define 
\[X=X_0^2=\{gh \mid g,h\in X_0\}.\]
The Cayley graph $\Gamma(G,X)$, as defined by (\ref{Eq: Cayley graph}), has no loops or multiple edges and is a 2-vertex-connected locally finite hyperbolic graph. Let $Y$ be the barycentric subdivision of $\Gamma(G,X)$. Note that $Y$ is fine, $G$ acts on $Y$ without inversion of edges, and all edge stabilizers are trivial. By Theorem \ref{Thm:bicombing}, there exists a $G$-equivariant anti-symmetric $\QQ$-bicombing $q$ of $Y$ and a constant $T\ge 0$ such that for any $a,b,c\in V(Y)$, we have
\[\|q[a,b]+q[b,c]+q[c,a]\|_1\le T.
\]
We define a map $Q\colon G \to \ell^2(Y)$ by (see Definition \ref{Def:modify} for details of this definition)
\[Q(g)=\widetilde{q[1,g]}.
\]
It is straightforward to check that $Q$ is an array into $(\ell^2(Y),\pi)$ (cf. Definition \ref{Def:unitary rep}) by using Lemma \ref{Lem:l1 to l2} (2). Also, we have
\[d_Y(1,g)\le\|q[1,g]\|_1=\|Q(g)\|_2^2
\]
(see Remark \ref{Rem:convex}, Lemma \ref{Lem:linear combination of paths}, and Lemma \ref{Lem:l1 to l2} (1)). Since $Y$ is a barycentric subdivision of $\Gamma(G,X)$, we have $d_Y(1,g)=2d_X(1,g)$ for any $g\in G$, where $d_X$ is a word metric on $G$ with respect to $X$. Hence, for any $N\in \NN$, if $g\in G$ satisfies $\|Q(g)\|_2\le N$, we have
\[d_X(1,g)=\frac{1}{2}d_Y(1,g) \le \frac{1}{2}\|Q(g)\|_2^2 \le \frac{1}{2}N^2.
\]
This implies that $Q$ is proper. Combined with Proposition \ref{Prop:equiv cond of bi-exact} (3), this gives another proof of the fact that hyperbolic groups are bi-exact. Here, note that $(\ell^2(Y),\pi)$ is a direct sum of copies $(\ell^2(G),\lambda_G)$, because $G$ acts on $Y$ without inversion of edges and its action on $E(Y)$ is free. Hence, it is weakly contained by $(\ell^2(G),\lambda_G)$. We can also use Mineyev's bicombing in \cite{Mi} instead of Theorem \ref{Thm:bicombing} in this case.
\par

Second, suppose that $G$ is a free product of bi-exact groups $H_1$ and $H_2$. In this case, $\{H_1,H_2\}$ are peripheral subgroups and we can construct a proper array on $G$ using the normal forms of elements of the free product and proper arrays on $H_1$ and $H_2$ as follows. Let $r_i\colon H_i\to \ell^2(H_i)$ be a proper array into the left regular representation $(\ell^2(H_i),\lambda_{H_i})$ for each $i=1,2$. We can assume $\|r_i(a)\|_2\ge 1$ for any $a\in H_i \setminus \{1\}$ by changing values of $r_i$ on a finite subset of $H_i$, if necessary. Also, we regard each $r_i$ as a map from $H_i$ to $\ell^2(G)$ by composing it with the embedding $\ell^2(H_i) \inj \ell^2(G)$. For $g\in G\setminus\{1\}$, without loss of generality, let $g=h_1k_1\cdots h_nk_n$ be the normal form of $g$, where $h_1,\cdots,h_n\in H_1\setminus\{1\}$ and $k_1,\cdots,k_n\in H_2\setminus\{1\}$. We define a map $R_1\colon G\to \ell^2(G)$ by
\[R_1(g)=\lambda_G(1)r_1(h_1)+\lambda_G(h_1k_1)r_1(h_2)+\cdots+\lambda_G(h_1k_2\cdots h_{n-1}k_{n-1})r_1(h_n).
\]
Here, we define $R_1(1)=0$. In the same way, we define a map $R_2\colon G\to \ell^2(G)$ from $r_2$. It is not difficult to show that the map $R\colon G \to \ell^2(G) \oplus \ell^2(G)$ defined by
\[R(g)=(R_1(g),R_2(g))
\]
is a proper array into $(\ell^2(G) \oplus \ell^2(G),\lambda_G\oplus\lambda_G)$.
\par

The case of general relatively hyperbolic groups is a combination of the above two cases. To every finitely generated relatively hyperbolic group, we associate two arrays. The first one is constructed in Section \ref{Subsection:first array} starting with a fine hyperbolic graph as in the case of a hyperbolic group above. The second array is a generalized version of the array $R$ for $G=H_1*H_2$. Hull-Osin's separating cosets play the same role as syllables in the normal form in the free product case. One technical remark is that the condition `$\|r_i(a)\|_2 \ge 1\; \forall a\in H_i \setminus\{1\}$' in the free product case is used to prevent the set $\{g\in G \mid \|R(g)\|=0\}$ from becoming infinite, but in the proof of Proposition \ref{Prop:main result}, we don't need this condition thanks to the first array.

\subsection{First array}\label{Subsection:first array}

This section is basically a continuation of Section \ref{Section:MY}. Notation and terminology related to graphs follow those in Section \ref{Section:MY}. The goal of this section is to prove Proposition \ref{Prop:first array} by using Mineyev-Yaman's bicombing.

\begin{prop}\label{Prop:first array}
Suppose that $G$ is a finitely generated group hyperbolic relative to a collection of subgroups $\{H_\lambda\}_{\lambda\in\Lambda}$ of $G$. Then, there exist a finite generating set $X$ of $G$, a 2-vertex-connected fine hyperbolic graph $Y$ on which $G$ acts without inversion of edges, and an array $Q\colon  G \to \ell^2(Y)$ into $(\ell^2(Y),\pi)$ that satisfy the following conditions.
\begin{itemize}
    \item[(1)]
    The edge stabilizer is trivial for any edge in $E(Y)$.
    \item[(2)]
    For any $g\in G$, we have
    \begin{equation}\label{Eq:first array}
        d_{\widehat{\Gamma}}(1,g)\le \frac{1}{2}\|Q(g)\|_2^2,
    \end{equation}
where $\widehat{\Gamma}$ is the coned-off Cayley graph of $G$ with respect to $X$.
\end{itemize}
\end{prop}

\begin{rem}\label{Rem:weak containment}
The unitary representation $(\ell^2(Y),\pi)$ in Proposition \ref{Prop:first array} is weakly contained by the left regular representation $(\ell^2(G),\lambda_G)$. Indeed, since $G$ acts on $Y$ without inversion of edges and all edge stabilizers are trivial, $(\ell^2(Y),\pi)$ is a direct sum of copies of $(\ell^2(G),\lambda_G)$.
\end{rem}

We first prove a few general lemmas about graphs. Lemma \ref{Lem:linear combination of paths} can be proven for graphs with loops and multiple edges as well exactly in the same way, but we stick to our current setting.

\begin{lem}\label{Lem:linear combination of paths}
Suppose that $Y$ is a connected graph without loops or multiple edges. If $a,b \in V(Y)$ are two vertices, $(p_j)_{j=1}^N \subset C_1(Y)$ are paths from $a$ to $b$ as 1-chains, and $(\alpha_j)_{j=1}^N\subset\CC$ are complex numbers, then we have
$$\left|\sum_{j=1}^N \alpha_j\right| \cdot d_Y(a,b)
\le
\left\|\sum_{j=1}^N \alpha_j p_j\right\|_1.$$
\end{lem}

\begin{proof}
We have $|d_Y(o(e),a)-d_Y(t(e),a)| \leq d_Y(o(e),t(e)) \leq 1$ for any $e=(u,v) \in E(Y)$, where we define $o(e)=u$ and $t(e)=v$. Hence, by defining 
\[C= \{ e\in E(Y)\mid d_Y(o(e),a)=d_Y(t(e),a) \}\]
and
    \begin{align*}
          D_n=\{&e\in E(Y)\mid d_Y(o(e),a)=n-1 \wedge d_Y(t(e),a)=n\} \\
          &\cup\{e\in E(Y)\mid d_Y(t(e),a)=n-1 \wedge d_Y(o(e),a)=n\}
    \end{align*}
 for each $n \in \NN$,
we get the decomposition of edges $E(Y)=C \sqcup \left( \bigsqcup_{n \in \NN} D_n \right)$.
Also, we can choose directions of edges so that they point outward from $a$, that is, there exists a set of positive edges $E^+(Y)$ such that
\[ D_n \cap E^+(Y) = \{e\in E(Y)\mid d_Y(o(e),a)=n-1\wedge d_Y(t(e),a)=n\} \] and
\[ D_n \cap E^-(Y) = \{e\in E(Y)\mid d_Y(t(e),a)=n-1\wedge d_Y(o(e),a)=n\} \]
for all $ n\in \NN$, where $E^-(Y)=\overline{E^+(Y)}$. For simplicity, we use the notation $C^+= C \cap E^+(Y)$, $D_n^+= D_n \cap E^+(Y)$, and $D_n^-= D_n \cap E^-(Y)$. Given a 1-chain $\xi = \sum_{e\in E^+(Y)} c_e [e]\in C_1(Y)$, we define linear maps
$Pr_{C^+},Pr_{D_n^+}\colon C_1(Y)\to C_1(Y)\;(n\in\NN)$ by
\[ Pr_{C^+}(\xi) = \sum_{e\in C^+} c_e [e]\;\;\;{\rm and}\;\;\;
Pr_{D_n^+}(\xi) = \sum_{e\in D_n^+} c_e [e]. \]
By $E^+(Y)=C^+ \sqcup \left( \bigsqcup_{n \in \NN} D^+_n \right)$, we have for any $\xi\in C_1(Y)$,
\[\xi= Pr_{C^+}(\xi) + \sum_{n\in \NN} Pr_{D_n^+}(\xi)\;\;\;{\rm and}\;\;\; 
\|\xi\|_1 = \|Pr_{C^+}(\xi)\|_1 + \sum_{n\in \NN}\|Pr_{D_n^+}(\xi)\|_1.\]
In particular, we have $\left\|\sum_{j=1}^N \alpha_j p_j\right\|_1 = \left\|Pr_{C^+}(\sum_{j=1}^N \alpha_j p_j)\right\|_1 + \sum_{n\in \NN}\left\|Pr_{D_n^+}(\sum_{j=1}^N \alpha_j p_j)\right\|_1$.\par
We will show that, for any $n \in \NN$ with $1\le n \le d_Y(a,b)$, we have
\[\Bigg|\sum_{j=1}^N \alpha_j\Bigg|
\le
\Bigg\|Pr_{D_n^+}(\sum_{j=1}^N \alpha_j p_j)\Bigg\|_1.\]
By identifying $\ell^0(Y)$ with $\CC^{E^+(Y)}$ as explained in Section \ref{Section:MY}, we define a linear map $\psi\colon C_1(Y) \to \CC$ by $\psi(\sum_{e\in E^+(Y)} c_e [e]) = \sum_{e\in E^+(Y)} c_e$. Note that the sums are finite and $\psi$ depends on the chosen set $E^+(Y)$ of positive edges. For any $\xi = \sum_{e\in E^+(Y)} c_e [e]\in C_1(Y)$, we have 
\[|\psi(\xi)| = \left|\sum_{e\in E^+(Y)} c_e\right| \le \sum_{e\in E^+(Y)} |c_e| = \|\xi\|_1.\]
Here, for any path $p_j$ from $a$ to $b$ and any $n \in \NN$ with $1\le n \le d_Y(a,b)$, we have
\[\psi(Pr_{D_n^+} (p_j))=1.\] 
Indeed, let $p_j = (e_1,\cdots,e_k)$ as a sequence of edges, where $o(e_1)=a$ and $t(e_k)=b$, and $p_j \cap D_n = (e_{i_1},\cdots,e_{i_m})$ as its subsequence by abuse of notation, i.e. $p_j=\sum_{l=1}^k [e_l]$ and $Pr_{D_n^+} (p_j) = \sum_{l=1}^m [e_{i_l}]$ as a 1-chain. We can see that $m$ is odd, i.e. $m=2l-1$ with $l\in\NN$ and also $e_{i_1} \in D_n^+,e_{i_2} \in D_n^-,e_{i_3} \in D_n^+, e_{i_4} \in D_n^-,\cdots,e_{i_{2l-1}} \in D_n^+$. This implies
\begin{align*}
    \psi(Pr_{D_n^+} p_j) & = \psi([e_{i_1}]) + \psi([e_{i_2}]) + \psi([e_{i_3}]) + \psi([e_{i_4}])+ \cdots + \psi([e_{i_{2l-1}]}) \\
   & = 1 + (-1) + 1 + (-1) + \cdots + 1 \\
   & =1.
\end{align*}
Hence,
\begin{align*}
    \Bigg\|Pr_{D_n^+}(\sum_{j=1}^N \alpha_j p_j)\Bigg\|_1 &\ge \Bigg|\psi(Pr_{D_n^+}(\sum_{j=1}^N \alpha_j p_j))\Bigg|
=\Bigg|\psi(\sum_{j=1}^N \alpha_j Pr_{D_n^+}(p_j))\Bigg|
=\Bigg|\sum_{j=1}^N \alpha_j \psi(Pr_{D_n^+}(p_j))\Bigg| \\
&=\Bigg|\sum_{j=1}^N \alpha_j\Bigg|.
\end{align*}
Thus, we finally get
\begin{align*}
    \Bigg\|\sum_{j=1}^N \alpha_j p_j\Bigg\|_1 
    &= \Bigg\|Pr_{C^+}(\sum_{j=1}^N \alpha_j p_j)\Bigg\|_1 + \sum_{n\in \NN}\Bigg\|Pr_{D_n^+}(\sum_{j=1}^N \alpha_j p_j)\Bigg\|_1 \\
    &\ge 0 + \sum_{n=1}^{d_Y(a,b)} \Bigg|\sum_{j=1}^N \alpha_j\Bigg|= \Bigg|\sum_{j=1}^N \alpha_j\Bigg|\cdot d_Y(a,b).
\end{align*}
\end{proof}

\begin{defn}\label{Def:modify}
Suppose that $Y$ is a graph without loops or multiple edges and $E^+(Y)$ is a set of positive edges. For $\xi = \sum_{e \in E^+(Y)} c_e [e] \in\ell^1(Y)$ such that $c_e\in\RR$ for any $e\in E^+(Y)$, we define an element $\widetilde{\xi} = \sum_{e \in E^+(Y)} b_e [e]\in \ell^2(Y)$ by
\[b_e =
\begin{cases}
\sqrt{c_e} & \mathrm{if}\;\; c_e \ge 0 \\
-\sqrt{|c_e|} & \mathrm{if}\;\; c_e < 0.
\end{cases}\]
\end{defn}

Note that the above definition is independent of the choice of a set of positive edges.

\begin{lem}\label{Lem:l1 to l2}
For any  $\xi_1,\xi_2 \in\ell^1(Y)$ whose coefficients are real, the following hold.
\begin{itemize}
    \item[(1)]
    $\|\widetilde{\xi_1}\|_2^2 = \|\xi_1\|_1$.
    \item[(2)]
    $\|\widetilde{\xi_1} - \widetilde{\xi_2}\|_2^2 \le 2\|\xi_1 - \xi_2\|_1$.
\end{itemize}
\end{lem}

\begin{proof}
(1) follows trivially from Definition \ref{Def:modify}. We will prove (2). For any $c_1,c_2\in\RR_{\ge0}$, we have \[|\sqrt{c_1}-\sqrt{c_2}|^2\le|\sqrt{c_1}-\sqrt{c_2}|(\sqrt{c_1} +\sqrt{c_2})=|c_1-c_2|,
\]
\[|\sqrt{c_1}+\sqrt{c_2}|^2 = 2(c_1+c_2)-(\sqrt{c_1}-\sqrt{c_2})^2 \le 2|c_1+c_2|.
\] 
Therefore, given
\[\xi_1 = \sum_{e \in E^+(Y)} c_{1e}[e],\;
\xi_2 = \sum_{e \in E^+(Y)} c_{2e}[e],\;
\widetilde{\xi_1} = \sum_{e \in E^+(Y)} b_{1e}[e],\;
\widetilde{\xi_2} = \sum_{e \in E^+(Y)} b_{2e}[e],\]
we have
$|b_{1e} - b_{2e}|^2 \le 2|c_{1e}-c_{2e}|$ for any $e \in E^+(Y)$.
\end{proof}

\begin{lem}\label{Lem:2-vertex-connected}
If $G$ is a group and $X_0$ is a generating set of $G$ such that $X_0=X_0^{-1}$ and $1\in X_0$, then the Cayley graph $\Gamma(G,X_0^2)$ is 2-vertex-connected, where
\[X_0^2=\{gh\in G \mid g,h\in X_0\}.\]
\end{lem}

\begin{proof}
Note that $X_0^2$ generates $G$ since $X_0 \subset X_0^2$. Since $G$ acts transitively on $V(\Gamma(G,X_0^2))$ by graph automorphisms, it's enough to prove that $\Gamma(G,X_0^2) \setminus \{1\}$ is connected. Any two vertices $g,h\in X_0 \setminus\{1\}$ are connected in $\Gamma(G,X_0^2) \setminus \{1\}$ by an edge having label $g^{-1}h\in X_0^2$. Also, for any vertex $gh \in X_0^2 \setminus\{1\}$, if $g\neq 1$, then $gh$ is connected to a vertex $g$ in $\Gamma(G,X_0^2) \setminus \{1\}$ by an edge of label $h^{-1}\in X_0$. Hence, any two vertices adjacent to $1$ are connected in $\Gamma(G,X_0^2) \setminus \{1\}$ by a path of length at most 3. This implies that $\Gamma(G,X_0^2) \setminus \{1\}$ is connected.
\end{proof}

In the following proof, recall that our coned-off Cayley graph is a graph without loops or multiple edges (see Definition \ref{coned off cayley graph}).

\begin{proof}[Proof of Proposition \ref{Prop:first array}]
We take a symmetric finite generating set $X_0$ of $G$ with $1\in X_0$, and define $X=X_0^2$. The Cayley graph $\Gamma(G,X)$ is 2-vertex-connected by Lemma \ref{Lem:2-vertex-connected}. Without loss of generality, we can assume that the subgroups $H_\lambda$ are non-trivial. The coned-off Cayley graph $\widehat{\Gamma}$ with respect to $X$ is 2-vertex-connected, because $\Gamma(G,X)$ is 2-vertex-connected and $\{H_\lambda\}_{\lambda\in\Lambda}$ doesn't contain the trivial subgroup. We denote by $Y$ the barycentric subdivision of $\widehat{\Gamma}$. $G$ acts on $Y$ without inversion of edges. Since $\widehat{\Gamma}$ is a 2-vertex-connected fine hyperbolic graph, so is $Y$. Here, we used the fact that $\widehat{\Gamma}$ doesn't have loops to ensure $Y$ is 2-vertex-connected. Also, since the number of $G$-orbits of $E(\widehat{\Gamma})$ is finite and the edge stabilizer is trivial for any edge in $E(\widehat{\Gamma})$, the action of $G$ on $Y$ satisfies these conditions as well. Therefore, by Theorem \ref{Thm:bicombing}, there exist a $G$-equivariant anti-symmetric $\QQ$-bicombing $q$ of $Y$ and a constant $T\ge 0$ such that for any $a,b,c\in V(Y)$, we have
\[\|q[a,b]+q[b,c]+q[c,a]\|_1 \le T.
\]
We define a map $Q\colon  G\to \ell^2(Y)$ by
\[Q(g)=\widetilde{q[1,g]}.
\]
For every $g\in G$, we have
 \[\pi_gQ(g^{-1})=\widetilde{q[g,1]}=-\widetilde{q[1,g]}=-Q(g),\]
because $q$ is $G$-equivariant and anti-symmetric.
Given any $g,h\in G$, Lemma \ref{Lem:l1 to l2} (2) implies
\begin{align*}
    \|Q(gh)-\pi_g(Q(h))\|_2^2
&=\|\widetilde{q[1,gh]}-\widetilde{q[g,gh]}\|_2^2
\le 2\|q[1,gh]-q[g,gh]\|_1 \\
&\le 2(\|q[1,gh]+q[gh,g]+q[g,1]\|_1+\|q[1,g]\|_1) \\
&\le 2(T+\|q[1,g]\|_1).
\end{align*}
Hence, $Q$ is an array into $(\ell^2(Y),\pi)$. 
By Remark \ref{Rem:convex}, for any $g\in G$, the 1-chain $q[1,g]$ is a convex combination of paths from $1$ to $g$ i.e. there exist paths $p_1,\cdots,p_N$ from $1$ to $g$ and $\alpha_1,\cdots,\alpha_N\in\QQ_{\ge0}$ with $\sum_{j= 1}^N \alpha_j=1$ such that
\[q[1,g]=\sum_{j= 1}^N \alpha_jp_j.
\]
Hence, by Lemma \ref{Lem:linear combination of paths} and Lemma \ref{Lem:l1 to l2} (1), we have
\[d_Y(1,g)=\left|\sum_{j=1}^N \alpha_j\right| \cdot d_Y(1,g)\le 
\left\|\sum_{j=1}^N \alpha_jp_j\right\|_1=\|q[1,g]\|_1=\|Q(g)\|_2^2.
\]
Since $Y$ is a barycentric subdivision of $\widehat{\Gamma}$, we have for any $g\in G$,
\[2d_{\widehat{\Gamma}}(1,g)=d_Y(1,g),\]
hence
\[
d_{\widehat{\Gamma}}(1,g)=\frac{1}{2}d_Y(1,g)\le
\frac{1}{2}\|Q(g)\|_2^2.
\]
\end{proof}

\begin{rem}\label{Rem:not proper}
Note that $Q$ is not even proper relative to $\{H_\lambda\}_{\lambda\in\Lambda}$. For example, if $G$ is a free product of infinite groups $H_1$ and $H_2$, then by Theorem \ref{Thm:bicombing} (4), we can show that there exists some $N''\in\NN$ such that 
$\{h_1h_2 \mid h_1\in H_1,\;h_2\in H_2\}\subset\{g\in G\mid\|Q(g)\|_2\le N''\}$.
\end{rem}

\subsection{Second array}\label{Subsection:second array}
This section is a continuation of Section \ref{Subsection:HO}. The goal of this section is to prove Proposition \ref{Prop:second array}. We will construct an array on $G$ from an array on a subgroup $H_\mu$ which is a member of a hyperbolicaly embedded collection of subgroups $\{H_\lambda\}_{\lambda\in\Lambda}$. The construction follows Section 4 of \cite{HO} and uses the notion of separating cosets explained in Section \ref{Subsection:HO}.

\begin{prop}\label{Prop:second array}
Suppose that $G$ is a group, $X$ is a subset of $G$, and $\{H_\lambda\}_{\lambda\in\Lambda}$ is a collection of subgroups hyperbolically embedded in $(G,X)$. Then, for any $\mu\in\Lambda$ and any array $r$ on $H_\mu$ into $(\ell^2(H_\mu) ,\lambda_{H_\mu})$, there exists an array $R$ on $G$ into $(\ell^2(G), \lambda_G)$ and a constant $K_\mu\ge 0$ satisfying the following$\colon$  for any $g\in G$, any separating coset $xH_\mu \in S_\mu(1,g;D)$, and any geodesic path $p$ in the relative Cayley graph $\Gamma(X\cup \mathcal{H})$ from $1$ to $g$, we have
\begin{equation}\label{Eq:second array}
    \|r(p_{in}(xH_\mu)^{-1}p_{out}(xH_\mu))\|_2 \le
\|R(g)\|_2+K_\mu.
\end{equation}
\end{prop}

The proof of Proposition \ref{Prop:second array} is essentially the same as the proof of Theorem 4.2 of \cite{HO}, but because we deal with arrays instead of quasi-cocycles, we give full details with all necessary changes to make the proof self-contained. \par

Suppose that $r\colon  H_\mu \to \ell^2(H_\mu)$ is an array on $H_\mu$ into the left regular representation $(\ell^2(H_\mu),\lambda_{H_\mu})$. By the embedding $\ell^2(H_\mu) \inj \ell^2(G)$, we can think of $r$ as a map $r\colon  H_\mu \to \ell^2(G)$. We define a map  $\widetilde{r}\colon G\times G\to \ell^2(G)$ by
\[\widetilde{r}(f,g)=
    \begin{cases}
    \lambda_G(f) r(f^{-1}g) & \mathrm{if}\;\; f^{-1}g\in H_\mu \\
    0 & \mathrm{if}\;\; f^{-1}g\notin H_\mu,
    \end{cases}
\]
where $(\ell^2(G),\lambda_G)$ is the left regular representation of $G$.

\begin{rem}
\label{support}
If $f,g\in G$ are in the same coset of $H_\mu$, i.e. there exists a $H_\mu$-coset $xH_\mu$ for some $x\in G$ such that $f,g\in xH_\mu$, then the support of $\widetilde{r}(f,g)$ is in $xH_\mu$.
\end{rem}

\begin{lem}\label{Lem:basic property}
For any $f,g,h\in G$, the following hold.
\begin{itemize}
    \item[(1)] 
    $\widetilde{r}(g,f)=-\widetilde{r}(f,g)$.
    \item[(2)] 
    $\widetilde{r}(hf,hg)=\lambda_G(h)\widetilde{r}(f,g)$.
\end{itemize}
\end{lem}

\begin{proof}
(1) For any $f,g\in G$, $f^{-1}g\in H_\mu$ if and only if $g^{-1}f\in H_\mu$. If $f^{-1}g\in H_\mu$, we have
\begin{align*}
\widetilde{r}(g,f)
&=\lambda_G(g) r(g^{-1}f) \\
&=\lambda_G(g)\left(-\lambda_{H_\mu}(g^{-1}f)r((g^{-1}f)^{-1})\right)
=-\lambda_G(f)r(f^{-1}g)
=-\widetilde{r}(f,g).
\end{align*}
\par
(2) If $f^{-1}g\in H_\mu$, we have
\[\widetilde{r}(hf,hg)=\lambda_G(hf) r((hf)^{-1}hg)
=\lambda_G(h)\lambda_G(f)r(f^{-1}g)
=\lambda_G(h)\widetilde{r}(f,g).
\]
\end{proof}

\begin{lem}\label{Lem:bound of array}
For any $g\in H_\mu$, we have
    \[\sup_{h\in H_\mu}\|\widetilde{r}(g,h)-\widetilde{r}(1,h)\|_2= 
    \sup_{h\in H_\mu}\|\widetilde{r}(1,h)-\widetilde{r}(1,hg)\|_2<\infty.\]
\end{lem}

\begin{proof}
For any $g\in H_\mu$, we have
\begin{align*}
    \sup_{h\in H_\mu}\|\widetilde{r}(g,h)-\widetilde{r}(1,h)\|_2
&=\sup_{h\in H_\mu}\|\lambda_{H_\mu}(g)r(g^{-1}h)-r(h)\|_2 \\
&=\sup_{h\in H_\mu}\|r(g^{-1}h)-\lambda_{H_\mu}(g^{-1})r(h)\|_2<\infty,
\end{align*}
and by Lemma \ref{Lem:basic property},
\begin{align*}
    \sup_{h\in H_\mu}\|\widetilde{r}(1,h)-\widetilde{r}(1,hg)\|_2
&=\sup_{h\in H_\mu}\|\lambda_{H_\mu}(h^{-1})\left(\widetilde{r}(1,h)-\widetilde{r}(1,hg)\right)\|_2 \\
&=\sup_{h\in H_\mu}\|\widetilde{r}(h^{-1},1)-\widetilde{r}(h^{-1},g)\|_2 \\
&=\sup_{h\in H_\mu}\|-\widetilde{r}(1,h^{-1})+\widetilde{r}(g,h^{-1})\|_2 \\
&=\sup_{h\in H_\mu}\|\widetilde{r}(g,h)-\widetilde{r}(1,h)\|_2.
\end{align*}
\end{proof}

In the following, we denote 
\begin{equation}\label{Eq:Kg}
K_g=\sup_{h\in H_\mu}\|\widetilde{r}(g,h)-\widetilde{r}(1,h)\|_2= 
    \sup_{h\in H_\mu}\|\widetilde{r}(1,h)-\widetilde{r}(1,hg)\|_2.
\end{equation}

\begin{rem}
For any $g\in H_\mu$, we have 
\begin{align*}
K_g
&=\sup_{h\in H_\mu}\|\widetilde{r}(1,h)-\widetilde{r}(1,hg)\|_2
=\sup_{h\in H_\mu}\|\widetilde{r}(1,hg^{-1})-\widetilde{r}(1,(hg^{-1})g)\|_2 \\
&=\sup_{h\in H_\mu}\|\widetilde{r}(1,h)-\widetilde{r}(1,hg^{-1})\|_2
=K_{g^{-1}}.
\end{align*}
\end{rem}

\begin{rem}\label{Rem:bound norm by K}
By Lemma \ref{Lem:basic property} (1), $\widetilde{r}(1,1)=0$, hence for any $g\in H_\mu$, we have
\[\|\widetilde{r}(1,g)\|_2=\|\widetilde{r}(g,1)-\widetilde{r}(1,1)\|_2
\le \sup_{h\in H_\mu}\|\widetilde{r}(g,h)-\widetilde{r}(1,h)\|_2
=K_g.
\]
\end{rem}

\begin{lem}\label{Lem:two arrays}
For any elements $f_1,f_2,g_1,g_2\in G$ that are in the same coset of $H_\mu$, we have
\[\|\widetilde{r}(f_1,g_1)-\widetilde{r}(f_2,g_2)\|_2\le K_{f_1^{-1}f_2}+K_{g_1^{-1}g_2}.\]
\end{lem}

\begin{proof}
By Lemma \ref{Lem:bound of array}, we have
\begin{align*}
    \|\widetilde{r}(f_1,&g_1)-\widetilde{r}(f_2,g_2)\|_2
    \le \|\widetilde{r}(f_1,g_1)-\widetilde{r}(f_2,g_1)\|_2
    +\|\widetilde{r}(f_2,g_1)-\widetilde{r}(f_2,g_2)\|_2 \\
    &\le \|\widetilde{r}(1,f_1^{-1}g_1)-\widetilde{r}(f_1^{-1}f_2,f_1^{-1}g_1)\|_2
    +\|\widetilde{r}(1,f_2^{-1}g_1)-\widetilde{r}(1,f_2^{-1}g_1\cdot g_1^{-1}g_2)\|_2 \\
    &\le K_{f_1^{-1}f_2}+K_{g_1^{-1}g_2}.
\end{align*}
\end{proof}

For $f,g\in G$ and $xH_\mu \in S_\mu(f,g;D)$, we define $\widetilde{R}(f,g;xH_\mu)\in \ell^2(G)$ by
\[\widetilde{R}(f,g;xH_\mu)=\frac{1}{|E(f,g;xH_\mu,D)|}\sum_{(u,v)\in E(f,g;xH_\mu,D)} \widetilde{r}(u,v).\]

\begin{rem}\label{Rem:support of RxH}
This is well-defined because $E(f,g;xH_\mu,D)$ is finite by Lemma \ref{Lem:3.8 of HO} (c). Also, the support of $\widetilde{R}(f,g;xH_\mu)$ is in $xH_\mu$ by Remark \ref{support}.
\end{rem}

\begin{lem}\label{Lem:property of RxH}
For any $f,g,h\in G$ and $xH_\mu\in S_\mu(f,g;D)$, the following holds.
\begin{itemize}
    \item [(a)]
    $\widetilde{R}(g,f;xH_\mu)=-\widetilde{R}(f,g;xH_\mu)$.
    \item [(b)]
    $\widetilde{R}(hf,hg;hxH_\mu)=\lambda_G(h)\widetilde{R}(f,g;xH_\mu)$.
\end{itemize}
\end{lem}

\begin{proof}
It follows from Lemma \ref{Lem:3.2 of HO}, Lemma \ref{Lem:3.8 of HO}, and Lemma \ref{Lem:basic property}.
\end{proof}

For $n\in \RR_\ge0$, we define a constant by
\begin{equation}\label{Eq:Kn}
K_n=\max\{K_g\mid g\in H_\mu \wedge \widehat{d}_\mu(1,g)\le n\},    
\end{equation}
where $K_g$ is defined by (\ref{Eq:Kg}). Because $(H_\mu,\widehat{d}_\mu)$ is a locally finite metric space (cf. Definition \ref{Def:hyp emb} (2)), we have $K_n<\infty$.

\begin{lem}\label{Lem:difference}
For any $f,g\in G$, any $xH_\mu \in S_\mu(f,g;D)$, and any $(u,v)\in E(f,g;xH_\mu,D)$, we have
\[\|\widetilde{R}(f,g;xH_\mu)-\widetilde{r}(u,v)\|_2\le 2K_D.\]
\end{lem}

\begin{proof}
For any $(u,v),(u',v')\in E(f,g;xH_\mu,D)$, we have $\widehat{d}_\mu(u,u')\le 3C\le D$ and $\widehat{d}_\mu(v,v')\le 3C\le D$ by Lemma \ref{Lem:3.3 of HO} (b). This implies, by Lemma \ref{Lem:two arrays},
\[\|\widetilde{r}(u,v)-\widetilde{r}(u',v')\|_2
    \le K_{u^{-1}u'}+K_{v^{-1}v'}\le K_D+K_D=2K_D.
\]
Thus, for any $(u,v)\in E(f,g;xH_\mu,D)$, we have
\begin{align*}
    \|\widetilde{R}(f,g;xH_\mu)-\widetilde{r}(u,v)\|_2
    &= \left\|\frac{1}{|E(f,g;xH_\mu,D)|}\sum_{(u',v')\in E(f,g;xH_\mu,D)} \left(\widetilde{r}(u',v')-\widetilde{r}(u,v)\right)\right\|_2 \\
    &\le \frac{1}{|E(f,g;xH_\mu,D)|}\sum_{(u',v')\in E(f,g;xH_\mu,D)} \left\|\widetilde{r}(u',v')-\widetilde{r}(u,v)\right\|_2 \\
    &\le \frac{1}{|E(f,g;xH_\mu,D)|}\sum_{(u',v')\in E(f,g;xH_\mu,D)}2K_D
    =2K_D.
\end{align*}
\end{proof}

Finally, we define a map $\widetilde{R}\colon G\times G\to \ell^2(G)$ by
\[\widetilde{R}(f,g)=\sum_{xH_\mu \in S_\mu(f,g;D)} \widetilde{R}(f,g;xH_\mu).\]
This implicitly means that if $S_\mu(f,g;D)$ is empty, then $\widetilde{R}(f,g)=0$.

\begin{lem}\label{Lem:norm of R}
For any $f,g,h\in G$, the following hold.
\begin{itemize}
    \item [(a)]
    $\widetilde{R}(g,f)=-\widetilde{R}(f,g)$.
    \item [(b)]
    $\widetilde{R}(hf,hg)=\lambda_G(h)\widetilde{R}(f,g)$.
    \item [(c)]
    $\|\widetilde{R}(f,g)\|_2^2=\sum_{xH_\mu \in S_\mu(f,g;D)} \|\widetilde{R}(f,g;xH_\mu)\|_2^2$.
\end{itemize}
\end{lem}

\begin{proof}
(a) and (b) follow from Lemma \ref{Lem:3.2 of HO} and Lemma \ref{Lem:property of RxH}. (c) follows from the fact that the support of $\widetilde{R}(f,g;xH_\mu)$ is in $xH_\mu$ as stated in Remark \ref{Rem:support of RxH}.
\end{proof}

For $g\in G$, we define
\begin{equation}\label{Eq:Lg}
L_g=\max\{\widehat{d}_\mu(u,v) \mid (u,v)\in E(1,g;xH_\mu,D),\; xH_\mu\in S_\mu(1,g;D)\}.
\end{equation}

$L_g\in\NN\cup\{0\}$ is well-defined by Corollary \ref{Cor:separating cosets are finite} and Lemma \ref{Lem:3.8 of HO} (c). \par
The proof of the following lemma is similar to Lemma 4.7 of \cite{HO}. 

\begin{lem}\label{Lem:bounded area}
For any $g\in G$, we have
\[\sup_{h\in G} \|\widetilde{R}(1,h)+\widetilde{R}(h,g)+\widetilde{R}(g,1)\|_2<\infty.\]
\end{lem}

\begin{proof}
For $g,h\in G$, suppose that $\widetilde{R}(1,h)+\widetilde{R}(h,g)+\widetilde{R}(g,1)
=\sum_{v\in G}\alpha_vv\in\ell^2(G)$. We define $\xi_{xH_\mu}=\sum_{v\in xH_\mu}\alpha_vv$ for each $H_\mu$-coset $xH_\mu$. Note that we have
\[\widetilde{R}(1,h)+\widetilde{R}(h,g)+\widetilde{R}(g,1)
=\sum_{xH_\mu\in G/H_\mu}\xi_{xH_\mu}.\] 
Let $S_\mu(1,h;D)=S_{1,h}'\sqcup S_{1,h}''\sqcup F_{1,h}$, $S_\mu(h,g;D)=S_{h,g}'\sqcup S_{h,g}''\sqcup F_{h,g}$, and $S_\mu(g,1;D)=S_{g,1}'\sqcup S_{g,1}''\sqcup F_{g,1}$ be the decomposition in Lemma \ref{Lem:3.9 of HO}. \par
If $xH_\mu \notin S_\mu(1,h;D)\cup S_\mu(h,g;D)\cup S_\mu(g,1;D)$, then we have $\xi_{xH_\mu}=0$ by Remark \ref{Rem:support of RxH}. \par
If $xH_\mu\in S_{1,h}'$, we have
\[\xi_{xH_\mu}=\widetilde{R}(1,h;xH_\mu)+\widetilde{R}(g,1;xH_\mu)
=\widetilde{R}(1,h;xH_\mu)-\widetilde{R}(1,g;xH_\mu)
=0
\]
since $S_{1,h}'\subset S_\mu(g,1;D)\setminus S_\mu(h,g;D)$ and
$E(1,h;xH_\mu,D)=E(1,g;xH_\mu,D)$. We can argue similarly for $S_{1,h}''$, $S_{h,g}'$, $S_{h,g}''$, $S_{g,1}'$, $S_{g,1}''$. Hence, we have
\[\widetilde{R}(1,h)+\widetilde{R}(h,g)+\widetilde{R}(g,1)=\sum_{xH_\mu\in F_{1,h}\cup F_{h,g}\cup F_{g,1}}\xi_{xH_\mu}.\] \par
If $xH_\mu\in F_{1,h}$, there are three cases to consider (see Figures \ref{Fig: triangle3}, \ref{Fig: triangle2}). We fix geodesic paths $p\in\mathcal{G}(1,h)$, $q\in\mathcal{G}(h,g)$, and $r\in\mathcal{G}(g,1)$. \par

\begin{figure}[htbp]

\begin{minipage}[c]{0.5\hsize}
\begin{center}
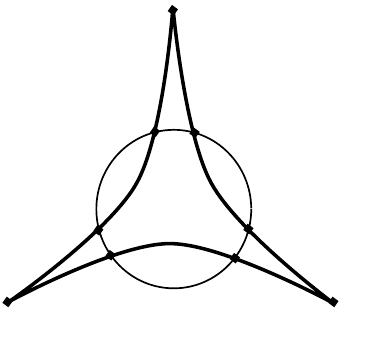
 \caption{Case 1, Case 2 a), Case 3 a)} 
 \label{Fig: triangle3}
\end{center}
\end{minipage} 
\hfill 
\begin{minipage}[c]{0.5\hsize}
\begin{center}
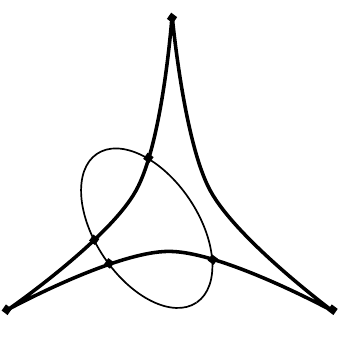
 \caption{Case 2 b), Case 3 b)}
 \label{Fig: triangle2}
\end{center}
\end{minipage} 
\end{figure}

\emph{Case 1}$\colon$ $xH_\mu \in S_\mu(h,g;D)\cap S_\mu(g,1;D)$. Let $a,b,c$ be the $H_\mu$-components of $p,q,r$ respectively, corresponding to $xH_\mu$, i.e. we have $p=p_1ap_2,~q=q_1bq_2,~r=r_1cr_2$. Let $e_1,e_2,e_3$ be paths of length at most 1 connecting $a_+$ to $b_-$, $b_+$ to $c_-$, $c_+$ to $a_-$ respectively, whose labels are in $H_\mu$. We claim that $e_1$ is isolated in the geodesic triangle $e_1q_1^{-1}p_2^{-1}$. Indeed, if $e_1$ is connected to an $H_\mu$-component $f$ of $p_2$ where we have $p_2=s_1 f s_2$, then there exists a path $t$ of length at most 1 connecting $a_-$ to $f_+$ since $a_-$ and $f_+$ are in the same $H_\mu$-coset. Hence, the path $p_1ts_2$ conneting $1$ to $h$ is shorter than $p$. This contradicts that $p$ is a geodesic. Similarly, $e_1$ is not connected to any $H_\mu$-component of $q_1$. Hence, we have $\widehat{d}_\mu(a_+,b_-)\le3C$ by Proposition \ref{Prop:2.4 of HO}. Similarly, we also have $\widehat{d}_\mu(b_+,c_-)\le3C,~\widehat{d}_\mu(c_+,a_-)\le3C$, hence
\[\widehat{d}_\mu(b_+,a_-)
\le \widehat{d}_\mu(b_+,c_-)+\widehat{d}_\mu(c_-,c_+)+\widehat{d}_\mu(c_+,a_-)
\le 3C+L_g+3C,
\]
where $L_g$ is defined by (\ref{Eq:Lg}).
Note that the inclusion $(c_-,c_+)\in E(g,1;xH_\mu,D)$ follows from the inclusions $r\in\mathcal{G}(g,1)$ and $xH_\mu\in S_\mu(g,1;D)$. Hence, by Lemma \ref{Lem:two arrays} (see also (\ref{Eq:Kn})), we have
\[\|\widetilde{r}(a_-,a_+)-\widetilde{r}(b_+,b_-)\|_2
\le K_{a_-^{-1}b_+}+K_{a_+^{-1}b_-}
\le K_{6C+L_g}+K_{3C}.
\]
Combining with Lemma \ref{Lem:difference} and Lemma \ref{Lem:norm of R} (c) (also note (\ref{Eq:D>3C})), we obtain
\begin{align*}
    \|\xi_{xH_\mu}\|_2
    &=\|\widetilde{R}(1,h;xH_\mu)+\widetilde{R}(h,g;xH_\mu)+\widetilde{R}(g,1;xH_\mu)\|_2 \\
    &\le\|\big(\widetilde{R}(1,h;xH_\mu)+\widetilde{R}(h,g;xH_\mu)\big)
    -\left(\widetilde{r}(a_-,a_+)+\widetilde{r}(b_-,b_+)\right)\|_2 \\
    &~~~~~~~~~~+\|\widetilde{r}(a_-,a_+)+\widetilde{r}(b_-,b_+)\|_2
    +\|\widetilde{R}(g,1;xH_\mu)\|_2 \\
    &\le 2K_D+2K_D+K_{6C+L_g}+K_{3C}+\|\widetilde{R}(g,1)\|_2 \\
    &\le 10K_{10D}+K_{6C+L_g}+\|\widetilde{R}(g,1)\|_2.
\end{align*}
\par
\emph{Case 2}$\colon$ $xH_\mu \in S_\mu(g,1;D)\setminus S_\mu(h,g;D)$ or $xH_\mu \in S_\mu(h,g;D)\setminus S_\mu(g,1;D)$. We can assume $xH_\mu \in S_\mu(g,1;D)\setminus S_\mu(h,g;D)$ without loss of generality. \par

2 a) If $q$ penetrates $xH_\mu$, then as in Case 1, let $a,b,c$ be the $H_\mu$-components of $p,q,r$ respectively, corresponding to $xH_\mu$ and $e_1,e_2,e_3$ be paths of length at most 1 connecting $a_+$ to $b_-$, $b_+$ to $c_-$, $c_+$ to $a_-$ respectively, whose labels are in $H_\mu$. In the same way as Case 1, we have $\widehat{d}_\mu(a_+,b_-)\le3C,~\widehat{d}_\mu(b_+,c_-)\le3C,~\widehat{d}_\mu(c_+,a_-)\le3C$. Also, by $xH_\mu\notin S_\mu(h,g;D)$, we have $\widehat{d}_\mu(b_-,b_+)\le D$, hence
\[\widehat{d}_\mu(a_+,c_-)
\le \widehat{d}_\mu(a_+,b_-)+\widehat{d}_\mu(b_-,b_+)+\widehat{d}_\mu(b_+,c_-)
\le 3C+D+3C \le3D.
\]
\par
2 b) If $q$ doesn't penetrate $xH_\mu$, let $a,c$ be the $H_\mu$-components of $p,r$ respectively, corresponding to $xH_\mu$, i.e. $p=p_1ap_2,~r=r_1cr_2$, and let $e_2,e_3$ be paths of length at most 1 connecting $a_+$ to $c_-$, $c_+$ to $a_-$ respectively, whose labels are in $H_\mu$. Because $e_3$ is isolated in the geodesic triangle $e_3p_1^{-1}r_2^{-1}$, we have $\widehat{d}_\mu(c_+,a_-)\le3C$. Also, because $q$ doesn't penetrate $xH_\mu$, $e_2$ is isolated in the geodesic 4-gon $e_2r_1^{-1}q^{-1}p_2^{-1}$, hence we have $\widehat{d}_\mu(a_+,c_-)\le4C\le3D$ by Proposition \ref{Prop:2.4 of HO}. \par

In both 2 a) and 2 b), we have $\widehat{d}_\mu(c_+,a_-)\le3C$ and $\widehat{d}_\mu(a_+,c_-)\le3D$. Thus, by Lemma \ref{Lem:two arrays} and \ref{Lem:difference}, we have
\begin{align*}
    \|\xi_{xH_\mu}\|_2 
    &=\|\widetilde{R}(1,h;xH_\mu)+\widetilde{R}(g,1;xH_\mu)\|_2 \\
    &\le \|\left(\widetilde{R}(1,h;xH_\mu)+\widetilde{R}(g,1;xH_\mu)\right)
    -\left(\widetilde{r}(a_-,a_+)+\widetilde{r}(c_-,c_+)\right)\|_2 \\
    &~~~~+\|\widetilde{r}(a_-,a_+)-\widetilde{r}(c_+,c_-)\|_2 \\
    &\le 2K_D+2K_D+ K_{3C}+K_{3D}\le 10K_{10D}.
\end{align*}
When $xH_\mu \in S_\mu(h,g;D)\setminus S_\mu(g,1;D)$, we can show $\|\xi_{xH_\mu}\|_2\le 10K_{10D}$ in the same way. \par

\emph{Case 3}$\colon$  $xH_\mu \notin S_\mu(h,g;D)\cup S_\mu(g,1;D)$. Note that by $xH_\mu\in F_{1,h}\subset S_\mu(1,h;D)$ and Lemma \ref{Lem:3.3 of HO} (a), at least one of $q$ and $r$ penetrates $xH_\mu$. \par

3 a) If both $q$ and $r$ penetrate $xH_\mu$, let $a,b,c,e_1,e_2,e_3$ be as in Case 1. Then, we have $\widehat{d}_\mu(a_+,b_-)\le3C,~\widehat{d}_\mu(b_+,c_-)\le3C,~\widehat{d}_\mu(c_+,a_-)\le3C$ in the same way as Case 1. Also, by $xH_\mu \notin S_\mu(h,g;D)\cup S_\mu(g,1;D)$, we have $\widehat{d}_\mu(b_-,b_+)\le D,~\widehat{d}_\mu(c_-,c_+)\le D$. Hence,
\begin{align*}
    \widehat{d}_\mu(a_+,a_-)
&\le \widehat{d}_\mu(a_+,b_-)+\widehat{d}_\mu(b_-,b_+)+\widehat{d}_\mu(b_+,c_-)
+\widehat{d}_\mu(c_-,c_+)+\widehat{d}_\mu(c_+,a_-) \\
&\le 3C+D+3C+D+3C \le5D.
\end{align*}
\par
3 b) If only one of $q$ and $r$ penetrates $xH_\mu$, assume $r$ penetrates $xH_\mu$ without loss of generality, and let $a,c,e_2,e_3$ be as in 2 b) of Case 2. Then, we have $\widehat{d}_\mu(c_+,a_-)\le3C$ and $\widehat{d}_\mu(a_+,c_-)\le4C$. Also by $xH_\mu \notin S_\mu(g,1;D)$, we have $\widehat{d}_\mu(c_-,c_+)\le D$. hence,
\begin{align*}
    \widehat{d}_\mu(a_+,a_-)
\le \widehat{d}_\mu(a_+,c_-)+\widehat{d}_\mu(c_-,c_+)+\widehat{d}_\mu(c_+,a_-)
\le 4C+D+3C \le5D.
\end{align*}
In both 3 a) and 3 b), we have $\widehat{d}_\mu(a_+,a_-)\le5D$, hence by Lemma \ref{Lem:two arrays} and Remark \ref{Rem:bound norm by K}, we have
\begin{align*}
    \|\xi_{xH_\mu}\|_2
    &=\|\widetilde{R}(1,h;xH_\mu)\|_2
    \le \|\widetilde{R}(1,h;xH_\mu)-\widetilde{r}(a_-,a_+)\|_2+\|\widetilde{r}(a_-,a_+)\|_2 \\
    &\le 2K_D+K_{5D} \le10K_{10D}.
\end{align*}

Here, we used $\|\widetilde{r}(a_-,a_+)\|_2=\|\widetilde{r}(1,a_-^{-1}a_+)\|_2
\le K_{a_-^{-1}a_+}\le K_{5D}$. When only $q$ penetrates $xH_\mu$, we can show $\|\xi_{xH_\mu}\|_2\le10K_{10D}$ in the same way. \par
Summarizing Case 1,2,3, if $xH_\mu\in F_{1,h}$, we have \[\|\xi_{xH_\mu}\|_2\le10K_{10D}+K_{6C+L_g}+\|\widetilde{R}(g,1)\|_2.\]
We can argue similarly for $F_{h,g}$ and $F_{g,1}$ as well. Also, by Lemma \ref{Lem:3.9 of HO} (c), we have $|F_{1,h}|,|F_{h,g}|,|F_{g,1}|\le2$. Thus, for any $g,h\in G$ we have
\begin{align*}
    \|\widetilde{R}(1,h)+\widetilde{R}(h,g)+\widetilde{R}(g,1)\|_2
&\le \sum_{xH_\mu\in F_{1,h}\cup F_{h,g}\cup F_{g,1}}\|\xi_{xH_\mu}\|_2 \\
&\le \sum_{xH_\mu\in F_{1,h}\cup F_{h,g}\cup F_{g,1}}
\left(10K_{10D}+K_{6C+L_g}+\|\widetilde{R}(g,1)\|_2\right)\\
&\le 6\left(10K_{10D}+K_{6C+L_g}+\|\widetilde{R}(g,1)\|_2\right).
\end{align*}
\end{proof}

\begin{proof}[Proof of Proposition \ref{Prop:second array}]
We define a map $R\colon G \to \ell^2(G)$ by
\[R(g)=\widetilde{R}(1,g).\]
By Lemma \ref{Lem:norm of R} (a) and (b), we have for any $g\in G$, \[\lambda_G(g)R(g^{-1})=\widetilde{R}(g,1)=-\widetilde{R}(1,g)=-R(g).\]
Also, by Lemma \ref{Lem:bounded area}, we have for any $g\in G$,
\begin{align*}
    \sup_{h\in G}\|R(gh)-\lambda_G(g)R(h)\|_2
    &=\sup_{h\in G}\|R(g\cdot g^{-1}h)-\lambda_G(g)R(g^{-1}h)\|_2 \\
    &=\sup_{h\in G}\|\widetilde{R}(1,h)-\widetilde{R}(g,h)\|_2 \\
    &\le \sup_{h\in G}\left(\|\widetilde{R}(1,h)+\widetilde{R}(h,g)+\widetilde{R}(g,1)\|_2+\|\widetilde{R}(1,g)\|_2\right) \\
    &=\|\widetilde{R}(1,g)\|_2
    +\sup_{h\in G}\|\widetilde{R}(1,h)+\widetilde{R}(h,g)+\widetilde{R}(g,1)\|_2 \\
    &<\infty.
\end{align*}
Thus, $R$ is an array. For any $g\in G$, any $xH_\mu \in S_\mu(1,g;D)$, and any geodesic path $p$ in $\Gamma(X\cup \mathcal{H})$ from $1$ to $g$, let $a$ be the component of $p$ corresponding to $xH_\mu$. By Lemma \ref{Lem:difference} and Lemma \ref{Lem:norm of R} (c), we have
\begin{align*}
\|r(a_-^{-1}a_+)\|_2
&=\|\widetilde{r}(a_-,a_+)\|_2 \\
&\le\|\widetilde{R}(1,g;xH_\mu)\|_2+2K_D \\
&\le\|\widetilde{R}(1,g)\|_2+2K_D=\|R(g)\|_2+2K_D,
\end{align*}
hence $R$ satisfies (\ref{Eq:second array}) with a constant $K_\mu=2K_D$.
\end{proof}

\subsection{Proof of main theorem}\label{Subsection:proof of main thm}
\begin{prop}\label{Prop:main result}
Suppose that $G$ is a finitely generated group hyperbolic relative to a collection of subgroups $\{H_\mu\}_{\mu \in \Lambda}$ of $G$. If all subgroups $H_\mu$ are bi-exact, then $G$ is also bi-exact.
\end{prop}

\begin{proof}
Note that $\Lambda$ is finite by definition. Because $H_\mu$'s are exact, $G$ is also exact by Corollary 3 of \cite{Oz1}. In the following, we will verify the condition of Proposition \ref{Prop:equiv cond of bi-exact} (3). We take a finite generating set $X$ of $G$, a unitary representation $(\ell^2(Y),\pi)$, and an array $Q$ as in Proposition \ref{Prop:first array}. Since every $H_\mu$ is bi-exact, there exists a proper array $r_\mu$ on $H_\mu$ into $(\ell^2(H_\mu),\lambda_{H_\mu})$ for each $\mu\in\Lambda$ by Proposition \ref{Prop:equiv cond of bi-exact} (2). By Proposition \ref{Prop:second array}, for each $r_\mu$, there exist an array $R_\mu$ on $G$ into $(\ell^2(G),\lambda_{G})$ and a constant $K_\mu \ge 0$ satisfying (\ref{Eq:second array}). Here, we used the fact that $\{H_\mu\}_{\mu \in \Lambda}$ is hyperbolically embedded in $(G,X)$ by Proposition \ref{Prop:hyp emb}. We define a Hilbert space 
$\mathcal{K}$ and a unitary representation $\rho$ of $G$ by
\begin{equation}\label{Eq:main rep}
  \mathcal{K}=\ell^2(Y) \oplus \Big(\bigoplus_{\mu\in\Lambda}\ell^2(G)\Big),~~~~~
\rho=\pi \oplus \Big(\bigoplus_{\mu\in\Lambda}\lambda_G \Big).  
\end{equation}
Since $(\mathcal{K},\rho)$ is a direct sum of copies of $(\ell^2(G),\lambda_G)$ by Remark \ref{Rem:weak containment}, $(\mathcal{K},\rho)$ is weakly contained by $(\ell^2(G),\lambda_G)$.
Now, we define a map
\begin{equation}\label{Eq:main array}
  P\colon G\ni g\mapsto \left(Q(g),(R_\mu(g))_{\mu\in\Lambda}\right) \in \ell^2(Y)\oplus \Big(\bigoplus_{\mu\in\Lambda}\ell^2(G) \Big)
  =\mathcal{K}.  
\end{equation}
Because $Q$ and $R_\mu$'s are arrays, $P$ is an array on $G$ into $(\mathcal{K},\rho)$. Hence, for any $g,h,k\in G$, we have
\begin{align*}
    \|P(kh)-\rho_gP(g^{-1}k)\| &\le\|P(kh)-P(k)\|+\|P(k)-\rho_gP(g^{-1}k)\| \\
    &=\|-\rho_{kh}P((kh)^{-1})+\rho_{k}P(k^{-1})\|+\|\rho_{g^{-1}}P(k)-P(g^{-1}k)\| \\
    &=\|-P(h^{-1}k^{-1})+\rho_{h^{-1}}P(k^{-1})\|+\|\rho_{g^{-1}}P(k)-P(g^{-1}k)\| \\
    &\le C(h^{-1})+C(g^{-1}),
\end{align*}
where we denote for each $s\in G$, 
\[
C(s)=\sup_{t\in G}\|P(st)-\rho_s(P(t))\|<\infty.
\]

Hence, for any $g,h\in G$, we have
\[\sup_{k\in G}\|P(gkh)-\rho_gP(k)\|
=\sup_{k\in G}\|P(g(g^{-1}k)h)-\rho_gP(g^{-1}k)\|\le C(h^{-1})+C(g^{-1}).\]
\par
Finally, we will show that $P$ is proper. Let $N\in\NN$ and $g\in G$ satisfy $\|P(g)\|\le N$. Since $\|Q(g)\|_2^2+\sum_{\mu\in\Lambda}\|R_\mu(g)\|_2^2=\|P(g)\|^2$, we get $\|Q(g)\|\le N$ and $\|R_\mu(g)\|\le N$ for any $\mu\in\Lambda$. Because the identity map $id\colon (G,d_{\widehat{\Gamma}})\to (G,d_{X\cup\mathcal{H}})$ is bi-Lipschitz, there exists $\alpha\in\NN$ such that $d_{X\cup\mathcal{H}}(1,x)\le \alpha d_{\widehat{\Gamma}}(1,x)$ for any $x\in G$. By (\ref{Eq:first array}), this implies 
\[
d_{X\cup\mathcal{H}}(1,g)\le \alpha d_{\widehat{\Gamma}}(1,g)\le \frac{1}{2}\alpha\|Q(g)\|_2^2\le \frac{1}{2}\alpha N^2.
 \]
We denote $\alpha_N=\frac{1}{2}\alpha N^2$ for simplicity. \par
Let $g=w_0h_1w_1\cdots h_nw_n$ be the label of a geodesic path from $1$ to $g$ in $\Gamma(G,X\cup\mathcal{H})$, where $w_i$ is a word in the alphabet $X\cup X^{-1}$ and $h_i\in \bigsqcup_{\mu\in\Lambda}(H_\mu\setminus\{1\})$. We have \[|w_0|+1+|w_1|+\cdots+1+|w_n|=d_{X\cup\mathcal{H}}(1,g)\le \alpha_N,\] 
where $|w_i|$ denotes the number of letters in $w_i$. In particular, we have $n\le \alpha_N$ and $|w_i|\le \alpha_N$ for any $w_i$. \par
On the other hand, for each $h_i$, there exists $\mu\in\Lambda$ such that $h_i\in H_\mu$. For simplicity, we denote
\[x_i=w_0h_1\cdots w_{i-1}.\] 
If $x_iH_\mu\notin S_\mu(1,g;D)$, then we have $\widehat{d}_\mu(1,h_i)=\widehat{d}_\mu(x_i,x_ih_i)\le D$ by definition of separating cosets. If $x_iH_\mu\in S_\mu(1,g;D)$, then by (\ref{Eq:second array}), we have
\begin{align*}
\|r_\mu(h_i)\|_2
\le\|R_\mu(g)\|_2+K_\mu
\le N+K_\mu.
\end{align*}
In either case, we have $h_i\in A_\mu$, where
\[A_\mu=\{h\in H_\mu \mid \widehat{d}_\mu(1,h)\le D\}\cup\{h\in H_\mu \mid \|r_\mu(h)\|_2\le N+K_\mu\}.\]
Note that $A_\mu$ is finite, because $\widehat{d}_\mu$ is a locally finite metric on $H_\mu$ and $r_\mu$ is proper. Therefore, we have
\[\{g\in G\mid\|P(g)\|\le N\}\subset
\Big\{w_0h_1w_1\cdots h_nw_n \;\Big|\; n\le \alpha_N,\; |w_i|\le \alpha_N,\; h_i\in \bigcup_{\mu\in\Lambda}A_\mu\Big\}.
\]
Since $X$, $A_\mu$'s, and $\Lambda$ are finite, the set on the right-hand side above is finite, hence $P$ is proper.
\end{proof}

\begin{lem}\label{Lem:subgroup of bi-exact group}
Subgroups of countable bi-exact groups are also bi-exact.
\end{lem}

\begin{proof}
Let $G$ be a countable bi-exact group and $H$ be a subgroup of $G$. $H$ is exact, because $G$ is exact and subgroups of exact groups are exact (cf. \cite{BO}). By Proposition \ref{Prop:equiv cond of bi-exact} (2), there exists a proper array $r\colon G\to\ell^2(G)$ on $G$ into $(\ell^2(G),\lambda_G)$. Note that the restriction of $\lambda_G$ to $H$, that is, $(\ell^2(G),\lambda_G|_H)$ is unitarily isomorphic to $(\bigoplus_{Hx\in H\backslash G} \ell^2(H),\lambda_H)$, hence we have $(\ell^2(G),\lambda_G|_H)\prec(\ell^2(H),\lambda_H)$. Also, it is straightforward to show that $r|_H\colon H\to \ell^2(G)$ is a proper array on $H$ into $(\ell^2(G),\lambda_G|_H)$. Thus, by Proposition \ref{Prop:equiv cond of bi-exact} (3), $H$ is bi-exact.
\end{proof}

\begin{proof}[Proof of Theorem \ref{Thm:main}]
It follows from Proposition \ref{Prop:main result} and Lemma \ref{Lem:subgroup of bi-exact group}.
\end{proof}



\begin{thebibliography}{10}

\bibitem{Bow}
B.~H. Bowditch, \emph{Relatively hyperbolic groups}, Internat. J. Algebra
  Comput. \textbf{22} (2012), no.~3, 1250016, 66. \MR{2922380}

\bibitem{BO}
N.~P. Brown and N.~Ozawa, \emph{{$C^*$}-algebras and finite-dimensional
  approximations}, Graduate Studies in Mathematics, vol.~88, American
  Mathematical Society, Providence, RI, 2008. \MR{2391387}

\bibitem{CI}
I.~Chifan and A.~Ioana, \emph{Amalgamated free product rigidity for group von
  {N}eumann algebras}, Adv. Math. \textbf{329} (2018), 819--850. \MR{3783429}

\bibitem{CSU}
I.~Chifan, T.~Sinclair, and B.~Udrea, \emph{On the structural theory of {$\rm
  II_1$} factors of negatively curved groups, {II}: {A}ctions by product
  groups}, Adv. Math. \textbf{245} (2013), 208--236. \MR{3084428}

\bibitem{Da}
F.~Dahmani, \emph{Les groupes relativement hyperboliques et leurs bords thèse
  de doctorat, université louis pasteur, strasbourg \rm{I}. (2003)},
  \url{https://www-fourier.ujf-grenoble.fr/~dahmani/Files/These.pdf}.

\bibitem{DGO}
F.~Dahmani, V.~Guirardel, and D.~Osin, \emph{Hyperbolically embedded subgroups
  and rotating families in groups acting on hyperbolic spaces}, Mem. Amer.
  Math. Soc. \textbf{245} (2017), no.~1156, v+152. \MR{3589159}

\bibitem{Hr}
G.~C. Hruska, \emph{Relative hyperbolicity and relative quasiconvexity for
  countable groups}, Algebr. Geom. Topol. \textbf{10} (2010), no.~3,
  1807--1856. \MR{2684983}

\bibitem{HO}
M.~Hull and D.~Osin, \emph{Induced quasicocycles on groups with hyperbolically
  embedded subgroups}, Algebr. Geom. Topol. \textbf{13} (2013), no.~5,
  2635--2665. \MR{3116299}

\bibitem{Mi}
I.~Mineyev, \emph{Straightening and bounded cohomology of hyperbolic groups},
  Geom. Funct. Anal. \textbf{11} (2001), no.~4, 807--839. \MR{1866802}

\bibitem{MY}
I.~Mineyev and A.~Yaman, \emph{Relative hyperbolicity and bounded cohomology},
  \url{https://faculty.math.illinois.edu/~mineyev/math/art/rel-hyp.pdf}.

\bibitem{Os1}
D.~V. Osin, \emph{Relatively hyperbolic groups: intrinsic geometry, algebraic
  properties, and algorithmic problems}, Mem. Amer. Math. Soc. \textbf{179}
  (2006), no.~843, vi+100. \MR{2182268}

\bibitem{Os2}
\bysame, \emph{Acylindrically hyperbolic groups}, Trans. Amer. Math. Soc.
  \textbf{368} (2016), no.~2, 851--888. \MR{3430352}

\bibitem{Oz00}
N.~Ozawa, \emph{Amenable actions and exactness for discrete groups}, C. R.
  Acad. Sci. Paris S\'{e}r. I Math. \textbf{330} (2000), no.~8, 691--695.
  \MR{1763912}

\bibitem{Oz3}
\bysame, \emph{Solid von {N}eumann algebras}, Acta Math. \textbf{192} (2004),
  no.~1, 111--117. \MR{2079600}

\bibitem{Oz1}
\bysame, \emph{Boundary amenability of relatively hyperbolic groups}, Topology
  Appl. \textbf{153} (2006), no.~14, 2624--2630. \MR{2243738}

\bibitem{Oz2}
\bysame, \emph{A {K}urosh-type theorem for type {$\rm II_1$} factors}, Int.
  Math. Res. Not. (2006), Art. ID 97560, 21. \MR{2211141}

\bibitem{Oz09}
\bysame, \emph{An example of a solid von {N}eumann algebra}, Hokkaido Math. J.
  \textbf{38} (2009), no.~3, 557--561. \MR{2548235}

\bibitem{OP}
N.~Ozawa and S.~Popa, \emph{Some prime factorization results for type {${\rm
  II}_1$} factors}, Invent. Math. \textbf{156} (2004), no.~2, 223--234.
  \MR{2052608}

\bibitem{PV}
S.~Popa and S.~Vaes, \emph{Unique {C}artan decomposition for {$\rm II_{1}$}
  factors arising from arbitrary actions of hyperbolic groups}, J. Reine Angew.
  Math. \textbf{694} (2014), 215--239. \MR{3259044}

\bibitem{Sako}
H.~Sako, \emph{Measure equivalence rigidity and bi-exactness of groups}, J.
  Funct. Anal. \textbf{257} (2009), no.~10, 3167--3202. \MR{2568688}

\end{thebibliography}

\providecommand{\bysame}{\leavevmode\hbox to3em{\hrulefill}\thinspace}
\providecommand{\MR}{\relax\ifhmode\unskip\space\fi MR }
\providecommand{\MRhref}[2]{%
  \href{http://www.ams.org/mathscinet-getitem?mr=#1}{#2}
}
\providecommand{\href}[2]{#2}

\vspace{5mm}

\noindent  Department of Mathematics, Vanderbilt University, Nashville 37240, U.S.A.

\noindent E-mail: \emph{koichi.oyakawa@vanderbilt.edu}

\end{document}